\documentclass[12pt]{amsart}
\usepackage[nobysame,abbrev,alphabetic]{amsrefs}
\usepackage{amssymb}
\usepackage[all]{xy}

\DeclareMathOperator{\Char}{char}

\DeclareMathOperator{\Ext}{\mathrm{Ext}}
\DeclareMathOperator{\Gal}{Gal}
\DeclareMathOperator{\Hom}{Hom}
\DeclareMathOperator{\id}{id}

\DeclareMathOperator{\Ker}{Ker}

\DeclareMathOperator{\pr}{pr}
\DeclareMathOperator{\res}{res}

\DeclareMathOperator{\Skew}{Skew}

\DeclareMathOperator{\trg}{trg}

\begin{document}

\newtheorem{thm}{Theorem}[section]
\newtheorem{cor}[thm]{Corollary}
\newtheorem{lem}[thm]{Lemma}
\newtheorem{prop}[thm]{Proposition}
\newtheorem{defin}[thm]{Definition}
\newtheorem{exam}[thm]{Example}
\newtheorem{examples}[thm]{Examples}
\newtheorem{rem}[thm]{Remark}
\newtheorem{case}{\sl Case}
\newtheorem{claim}{Claim}
\newtheorem{prt}{Part}
\newtheorem*{mainthm}{Main Theorem}
\newtheorem*{thmA}{Theorem A}
\newtheorem*{thmB}{Theorem B}
\newtheorem{question}[thm]{Question}
\newtheorem*{notation}{Notation}
\swapnumbers
\newtheorem{rems}[thm]{Remarks}
\newtheorem*{acknowledgment}{Acknowledgment}

\newtheorem{questions}[thm]{Questions}
\numberwithin{equation}{section}

\newcommand{\dec}{\mathrm{dec}}
\newcommand{\dirlim}{\varinjlim}
\newcommand{\discup}{\ \ensuremath{\mathaccent\cdot\cup}}
\newcommand{\divis}{\mathrm{div}}
\newcommand{\nek}{,\ldots,}
\newcommand{\inv}{^{-1}}
\newcommand{\isom}{\cong}
\newcommand{\ndiv}{{\not|}}
\newcommand{\sep}{\mathrm{sep}}
\newcommand{\sym}{\mathrm{sym}}
\newcommand{\tagg}{^{''}}
\newcommand{\tensor}{\otimes}

\newcommand{\alp}{\alpha}
\newcommand{\gam}{\gamma}
\newcommand{\Gam}{\Gamma}
\newcommand{\del}{\delta}
\newcommand{\Del}{\Delta}
\newcommand{\eps}{\epsilon}
\newcommand{\lam}{\lambda}
\newcommand{\Lam}{\Lambda}
\newcommand{\sig}{\sigma}
\newcommand{\Sig}{\Sigma}
\newcommand{\bfA}{\mathbf{A}}
\newcommand{\bfB}{\mathbf{B}}
\newcommand{\bfC}{\mathbf{C}}
\newcommand{\bfF}{\mathbf{F}}
\newcommand{\bfP}{\mathbf{P}}
\newcommand{\bfQ}{\mathbf{Q}}
\newcommand{\bfR}{\mathbf{R}}
\newcommand{\bfS}{\mathbf{S}}
\newcommand{\bfT}{\mathbf{T}}
\newcommand{\bfZ}{\mathbf{Z}}
\newcommand{\dbA}{\mathbb{A}}
\newcommand{\dbC}{\mathbb{C}}
\newcommand{\dbF}{\mathbb{F}}
\newcommand{\dbN}{\mathbb{N}}
\newcommand{\dbQ}{\mathbb{Q}}
\newcommand{\dbR}{\mathbb{R}}
\newcommand{\dbZ}{\mathbb{Z}}
\newcommand{\grf}{\mathfrak{f}}
\newcommand{\gra}{\mathfrak{a}}
\newcommand{\grm}{\mathfrak{m}}
\newcommand{\grp}{\mathfrak{p}}
\newcommand{\grq}{\mathfrak{q}}
\newcommand{\calA}{\mathcal{A}}
\newcommand{\calC}{\mathcal{C}}
\newcommand{\calE}{\mathcal{E}}
\newcommand{\calG}{\mathcal{G}}
\newcommand{\calL}{\mathcal{L}}
\newcommand{\calW}{\mathcal{W}}
\newcommand{\calV}{\mathcal{V}}

\title[On the descending central sequence]{On the descending central sequence of absolute Galois groups}
\author{Ido Efrat}
\address{Mathematics Department\\
Ben-Gurion University of the Negev\\
P.O.\ Box 653, Be'er-Sheva 84105\\
Israel}
\email{efrat@math.bgu.ac.il}

\author{J\'an Min\'a\v c}
\address{Mathematics Department\\
University of Western Ontario\\
London\\
Ontario\\
Canada N6A 5B7}
\email{minac@uwo.ca}
\thanks{The first author was supported by the Israel Science Foundation (grant No.\ 23/09).
The second author was supported in part by  National Sciences and Engineering Council of Canada grant R0370A01.}

\begin{abstract}
Let $p$ be an odd prime number and $F$ a field containing a primitive $p$th root of unity.
We prove a new restriction on the group-theoretic structure of the absolute Galois group $G_F$ of $F$.
Namely, the third subgroup $G_F^{(3)}$ in the descending $p$-central sequence of $G_F$ is the
intersection of all open normal subgroups $N$ such that $G_F/N$ is $1$, $\dbZ/p^2$,
or the extra-special group $M_{p^3}$ of order $p^3$ and exponent $p^2$.
\end{abstract}


\maketitle

\section{Introduction}
Let $q=p^d$ be a prime power, and let $G$ be a profinite group.
The \textbf{descending $q$-central sequence} of $G$ is defined inductively by
\[
G^{(1)}=G,\quad
G^{(i+1)}=(G^{(i)})^q[G^{(i)},G],\quad i=1,2,\ldots \ .
\]
Thus $G^{(i+1)}$ is the closed subgroup of $G$ generated by all powers $h^q$ and all commutators $[h,g]=h\inv g\inv hg$,
where $h\in G^{(i)}$ and $g\in G$.

Now suppose that $q=p$.
Let $F$ be  a field containing a primitive $p$th root of unity $\zeta_p$,
and let $G=G_F$ be its absolute Galois group.
Let $M_{p^3}$ be the unique nonabelian group of order $p^3$
and exponent $p^2$ (see \S8).

\begin{mainthm}
For $p\neq2$ and for $G=G_F$ as above, $G^{(3)}$ is the intersection of all open normal subgroups
$N$ of $G$ such that $G/N$ is isomorphic to one of $1$, $\dbZ/p^2$, and $M_{p^3}$.
\end{mainthm}

Determining the profinite groups which are realizable as absolute Galois groups of fields
is a major open problem in Galois theory.
Our Main Theorem appears to be simple yet powerful restriction on the possible structure of such groups,
and on their quotients $G_F/G_F^{(3)}$.
These quotients an extremely important invariant of fields, carrying a substantial information about
their arithmetical structure.
For example, when $p=2$ it encodes the orderings and the Witt ring of quadratic forms
of $F$ (\cite{MinacSpira90}, \cite{MinacSpira96}) as well as some non-trivial valuations \cite{MaheMinacSmith04}.
Further, it encodes the entire mod $2$ Galois cohomology ring of $G_F$ \cite{AdemKaraMinac99}*{Th.\ 3.14}.
In a joint work with S.\ Chebolu \cite{CheboluEfratMinac}, we show that $G_F/G_F^{(3)}$ can be in fact thought of as a group-theoretic
analog of the Galois cohomology of $F$ for any $p$, and we use this close relationship between
the Galois cohomology of $F$ and $G_F/G_F^{(3)}$ to provide
new examples of profinite groups which are not realizable as absolute Galois groups of fields.

The analog of our Main Theorem for $p=2$ was discovered by Villegas in a different formalism \cite{Villegas88}.
The second author and Spira reformulated and reproved it in \cite{MinacSpira96}*{Cor.\ 2.18}
using the descending $2$-central sequence of $G_F$.
Namely, then  $G^{(3)}=G_F^{(3)}$ is the intersection of all open normal subgroups $N$ of $G$
such that $G/N$ is isomorphic to $1$, $\dbZ/2$, $\dbZ/4$, or to the dihedral group $D_4=M_8$ of order $8$.

A main difference between the case $p>2$ and the case $p=2$ is the existence in the former case of elements in
$H^2((\dbZ/p)^n,\dbZ/p)$ which are not expressible as sums of cup products of
elements in $H^1((\dbZ/p)^n,\dbZ/p)$.
To handle this new kind of element we study the Bockstein homomorphism $\beta_G\colon H^1(G,\dbZ/p)\to  H^2(G,\dbZ/p)$
and its relation to Galois theory.

Our approach is purely cohomological.
Thus we prove the Main Theorem more generally for profinite groups $G$ which satisfy two simple conditions
on their lower cohomology.
These conditions are known to hold for $G=G_F$, with $F$ as above, where they are consequences of the
following two Galois-theoretic facts (see \S3 for details and terminology):
\begin{enumerate}
\item[(i)]
the Galois symbol $K^M_2(F)/p\to H^2(G,\dbZ/p)$ is injective (it is actually bijective by the
Merkurjev--Suslin theorem, which is a special case of the Rost--Voevodsky's theorem); and
\item[(ii)]
$\beta_G$ is the cup product by the Kummer element $(\zeta_p)\in H^1(G,\dbZ/p)$.
\end{enumerate}

More generally, when $q=p^d$ is an arbitrary prime power and $F$ is a field containing a primitive $q$th
root of unity, we
characterize $G_F^{(3)}$ as the intersection of all open normal subgroups $N$ of $G_F$ such that $G_F/N$
belongs to a certain
cohomologically defined class of finite groups (Theorem \ref{formal main theorem}).
This is based on the natural generalizations of (i) and (ii) above, as well as the following additional
property of $G_F$:
\begin{enumerate}
\item[(iii)]
the map $H^1(G_F,\dbZ/q)\to H^1(G_F,\dbZ/p^i)$, $1\leq i\leq d$, is surjective.
\end{enumerate}

Our analysis applies also to $p=2$.
Thus we give a new cohomological proof of the above-mentioned result of \cite{Villegas88} and \cite{MinacSpira96}
and  generalize it to profinite groups $G$ satisfying the appropriate conditions on their lower cohomology.
We also show that the group $\dbZ/2$ can be omitted from the list unless $F$ is a
Euclidean field (Corollary \ref{short list p even}).

The paper is organized as follows:
In \S2 we collect various cohomological preliminaries, especially facts related to the Bockstein
homomorphism $\beta_G$ and its connections with roots of unity and cup products.
In \S3 we introduce the key notion of a profinite group of Galois relation type.
This notion axiomatizes the cohomological properties of absolute Galois groups that we need for our proofs ((i)--(iii) above).
In \S4 we define an abelian group $\Omega(G)$ and a homomorphism $\Lambda_G\colon \Omega(G)\to
H^2(G,\dbZ/q)$.
These extend the cup product map $\cup\colon H^1(G,\dbZ/q)^{\tensor2}\to H^2(G,\dbZ/q)$ but also take into account
the Galois-theoretic role of  $\beta_G$.
Our axioms on $G$ imply that $\Ker(\Lambda_G)$ is generated by elements of simple type
(Definition \ref{simple type} and Proposition \ref{decomposition}).
These simple type elements are in turn related to cohomologically defined open subgroups
$N$ of $G$ of index dividing $q^3$, which we call ``distinguished subgroups".
In \S5 we translate the above result about $\Ker(\Lam_G)$ to the language of distinguished subgroups,
and prove the crucial Theorem \ref{formal main theorem}:
for $G$ of Galois relation type, $G^{(3)}$ is the intersection of all distinguished subgroups of $G$.

In \S\S6--10 we build a ``dictionary" between the images under $\Lam_G$ of simple type elements of
$\Omega(G)$ and some special group extensions.
The solutions of the resulting embedding problems correspond to distinguished subgroups of $G$.
This is then used in \S11 to prove the Main Theorem and the analogous results for $p=2$ in the
general setting of profinite groups of Galois relation type.

In \S12 we study $G/G^{(3)}$ for $G$ of Galois relation type.
As a corollary we recover some known ``automatic realization" results in Galois theory.
Our approach seems to provide a good explanation why these curious automatic realization
results are true.
Finally, in \S13 we give examples showing that all the finite groups in our lists are
indeed necessary.

We are grateful to the referee for his/her detailed comments and valuable suggestions.
We also thank  P.\ Deligne, L.\ Moret-Baily,  A.\  Shalev, and T.\ Weigel for their interest in this work
and their comments related to talks given at the Israel Mathematical Union 2008 conference and
the 2008 conference on Profinite Groups in ESI, Vienna.
Finally, we thank Y.\ Tschinkel for his kind encouragement.

\section{cohomological preliminaries}
Let $p$ be a prime number, let $q=p^d$ be a power of $p$, and let $G$ be a profinite group.
We write $H^i(G)$ for the profinite cohomology group $H^i(G,\dbZ/q)$, where $G$ acts trivially on $\dbZ/q$.
Thus $H^1(G)=\Hom(G,\dbZ/q)$ consists of all continuous group homomorphisms $G\to\dbZ/q$.
We consider $H^*(G)=\bigoplus_{i=0}^\infty H^i(G)$ as a graded anti-commutative ring with respect
to the cup product $\cup$.

\subsection*{A)\ Normal Subgroups}
Let $N$ be a normal closed subgroup of $G$.
Then $G$ acts canonically on $H^i(N)$.
Denote the group of all $G$-invariant elements of $H^i(N)$ by $H^i(N)^G$.
For $i=1$ this action is given by $\varphi\mapsto\varphi^g$, where $\varphi^g(n)=\varphi(g\inv ng)$
for $g\in G$ and $n\in N$.
Thus $H^1(N)^G$ consists of all homomorphisms $\varphi\colon N\to\dbZ/q$ which are trivial on $N^q[N,G]$.

The next lemma provides a fundamental connection between the descending $q$-central sequence of $G$ and cohomology.

\begin{lem}
\label{Pontryagin}
For a normal closed subgroup $N$ of $G$ one has
\[
\bigcap\{\Ker(\varphi)\ |\ \varphi\in H^1(N)^G \}=N^q[N,G].
\]
\end{lem}
\begin{proof}
Consider the natural projection  $\pi\colon N\to\bar N=N/N^q[N,G]$.
The abelian torsion group $\bar N$ has Pontryagin dual $H^1(\bar N)$.
By the Pontryagin duality \cite{NeukirchSchmidtWingberg}*{Th.\ 1.1.11},
$\bigcap_{\bar\varphi\in H^1(\bar N)}\Ker(\bar\varphi)=\{0\}$, whence
 $\bigcap_{\bar\varphi\in H^1(\bar N)}\pi\inv(\Ker(\bar\varphi))=N^q[N,G]$.
Further, if $\bar\varphi\in H^1(\bar N)$ and $\varphi=\inf_N(\bar\varphi)$,
then $\Ker(\varphi)=\pi\inv(\Ker(\bar\varphi))$.
Finally, by the previous remarks, $\textstyle{\inf_N}\colon H^1(\bar N)\to H^1(N)^G$ is an isomorphism.
The assertion follows.
\end{proof}

\begin{cor}
There is a natural non-degenerate pairing
\[
N/N^q[N,G]\times H^1(N)^G\to\dbZ/q.
\]
\end{cor}

\begin{cor}
$G^{(i)}/G^{(i+1)}$ is dual to $H^1(G^{(i)})^G$ for $i\geq1$.
\end{cor}

\subsection*{B)\ Spectral sequences}
Let $N$ be a closed normal subgroup of $G$.
Recall that the Hochschild--Serre spectral sequence
\[
E_2^{ij}=H^i(G/N,H^j(N))\Rightarrow H^{i+j}(G)
\]
induces the $5$-term exact sequence
\begin{equation}
\label{5-term sequence}
0\to H^1(G/N)\xrightarrow{\inf_G}   H^1(G) \xrightarrow{\res_N} H^1(N)^G \xrightarrow{\trg_{G/N}}
H^2(G/N)  \xrightarrow{\inf_G} H^2(G).
\end{equation}
Here  $\trg_{G/N}$ is the differential $d_2^{0,1}$ of the spectral sequence \cite{NeukirchSchmidtWingberg}*{Th.\ 2.4.3}.
If $N'$ is another closed normal subgroup of $G$ and $N'\leq N$, then the inflation maps induced by
the projection $G/N'\to G/N$
and the restriction map $\res_{N'}\colon H^j(N)\to H^j(N')$ induce a spectral sequence morphism from
$H^i(G/N,H^j(N))\Rightarrow H^{i+j}(G)$ to $H^i(G/N',H^j(N'))\Rightarrow H^{i+j}(G)$
\cite{NeukirchSchmidtWingberg}*{p.\ 98}.
In particular, there is a commutative diagram
\begin{equation}
\label{trg res and inf commute}
\xymatrix{
 H^1(N)^G \ar[r]^{\trg_{G/N}}\ar[d]_{\res_{N'}} & H^2(G/N)\ar[d]^{\inf_{G/N'}} \\
 H^1(N')^G \ar[r]^{\trg_{G/N'}} &  H^2(G/N')  .\\
}
\end{equation}

\subsection*{C)\ Connecting homomorphisms}
Let $n$ be a positive integer.
The exact sequence
\[
0\ \to \ \dbZ/n\ \to\  \dbZ/n^2\ \to\ \dbZ/n\ \to\ 0
\]
of trivial $G$-modules gives rise to a connecting homomorphism
\[
\beta_{G,n}\colon H^1(G,\dbZ/n)\to H^2(G,\dbZ/n).
\]
When $n=q$ is our fixed $p$-power, we abbreviate
\[
\beta_G=\beta_{G,q}
\]
and call it the \textbf{Bockstein homomorphism} of $G$.
Note that it is functorial in $G$.
We now relate $\beta_{G,n}$ to some other connecting homomorphisms and cup products.

\begin{lem}
\label{formula for Bockstein when q is 2}
Suppose $q=2$.
For $\psi\in H^1(G)$ one has $\beta_G(\psi)=\psi\cup\psi$.
\end{lem}
\begin{proof}
This is straightforward when $G\isom\dbZ/2$.
In the general case, it follows by inflating from $G/\Ker(\psi)$ to $G$.
\end{proof}

Next let $\eps\colon H^1(G,\dbQ/\dbZ)\to H^2(G,\dbZ)$ be the connecting map arising from the short
exact sequence of trivial $G$-modules
\[
0\to\dbZ\hookrightarrow\dbQ\to\dbQ/\dbZ\to0.
\]
Since $\dbQ$ is cohomologically trivial, $\eps$ is in fact an isomorphism.
Let
\[
j_n\colon \tfrac1n\dbZ/\dbZ\xrightarrow{\sim}\dbZ/n, \quad
\pi_n\colon\dbZ\to \dbZ/n
\]
be the natural maps.
A routine computation gives:

\begin{lem}
\label{Bockstein and epsilon}
One has $\beta_{G,n}\circ j_n^*=\pi_n^*\circ\eps$ on $H^1(G,\tfrac1n\dbZ/\dbZ)$.
\end{lem}

Now let $F$ be a field with $\Char\,F\ndiv\,  n$ and containing the group $\mu_n$ of all roots of unity
in $F_\sep$ of order dividing $n$.
Let
\[
\kappa_n\colon F^\times=H^0(G_F,F_\sep^\times)\to H^1(G_F,\mu_n).
\]
be the Kummer homomorphism.
We shall denote by same symbol $\kappa_n$ also the restricted map  $\mu_n=H^0(G_F,\mu_n)\to H^1(G_F,\mu_n)$.

\begin{prop}
\label{Bockstein and Kummer}
There is an equality of maps
\[
\beta_{G_F,n}\cup\id=\id\cup\kappa_n\colon
 H^1(G_F,\dbZ/n)\times H^0(G_F,\mu_n)\to H^2(G_F,\mu_n).
\]
\end{prop}
\begin{proof}
The embeddings $\dbZ/n\hookrightarrow\dbQ/\dbZ$ and $\mu_n\hookrightarrow F^\times$ induce a commutative diagram
\[
\xymatrix{
H^1(G_F,\dbZ/n)\times H^0(G_F,\mu_n)\ar[r]\ar[d]_{\beta_{G_F,n}\times\id} &
H^1(G_F,\dbQ/\dbZ)\times F^\times\ar[d]_{\wr}^{\eps\times\id\quad} \\
H^2(G_F,\dbZ/n)\times H^0(G_F,\mu_n)\ar[d]_{\cup} &
H^2(G_F,\dbZ)\times H^0(G_F,F_\sep^\times)\ar[d]^{\cup} \\
H^2(G_F,\mu_n) \ar[r] &  H^2(G_F,F_\sep^\times).
}
\]
We first show that this diagram commutes.

For this end, consider $\chi\in H^1(G_F,Z/n)$, $\zeta\in\mu_n$, and let $[\zeta]$ be the corresponding
element of $H^0(G_F,\mu_n)$.
Let $s$ be the section of the canonical map $\dbQ\to\dbQ/\dbZ$ given by $0\leq s(q+\dbZ)<1$ for $q\in \dbQ$.
Then $\eps(\chi)\in H^2(G_F,\dbZ)$ is represented by the $2$-cocycle $G_F\times G_F\to \dbZ$,
\begin{equation}
\label{cocycle1}
(\sig,\tau)\mapsto
\begin{cases}
0 & a+b<n\\  1& a+b\geq n,
\end{cases}
\end{equation}
where $\chi(\sig)=a$ (mod $n$), $\chi(\tau)=b$ (mod $n$) with $0\leq a,b<n$.
Hence $\eps(\chi)\cup[\zeta]\in H^2(G_F,F_\sep^\times)$ is represented by the $2$-cocycle
\begin{equation}
\label{cocycle2}
(\sig,\tau)\mapsto
\begin{cases}
1 & a+b<n\\  \zeta& a+b\geq n.
\end{cases}
\end{equation}

Likewise, let $t$ be the section of the canonical map $\dbZ/n^2\to\dbZ/n$ given for $0\leq i<n$  by $i$ (mod $n$) $\mapsto i$ (mod $n^2$).
Using it we similarly see that $\beta_{G_F,n}(\chi)\in H^2(G_F,\dbZ/n)$ is represented by the $2$-cocycle (\ref{cocycle1}) taken modulo $n$.
Hence $\beta_{G_F,n}(\chi)\cup[\zeta]\in H^2(G_F,\mu_n)$ is also represented by the $2$-cocycle (\ref{cocycle2}), and the commutativity follows.

Next we recall the identification of cyclic algebras as symbol algebras (for fields containing enough roots of unity), which in cohomological terms means that
the following diagram commutes (see \cite{SerreLF}*{Ch.\ 14, Prop.\ 5 and Remark 1}):
\[
\xymatrix{
H^1(G_F,\dbZ/n)\times F^\times\ar[r]\ar[d]_{\id\times\kappa_n} & H^1(G_F,\dbQ/\dbZ)\times F^\times\ar[d]^{\eps\times\id}_{\wr} \\
H^1(G_F,\dbZ/n)\times H^1(G_F,\mu_n)\ar[d]_{\cup}  & H^2(G_F,\dbZ)\times H^0(G_F,F_\sep^\times)\ar[d]^{\cup} \\
H^2(G_F,\mu_n)\ar[r] & H^2(G_F,F_\sep^\times).
}
\]

Finally, Hilbert's Satz 90 implies that the natural map $H^2(G_F,\mu_n)\to H^2(G_F,F_\sep^\times)$ is injective.
Identifying $H^0(G_F,\mu_n)$ as a subgroup of $F^\times$, we therefore conclude from the two diagrams that $\beta_{G_F,n}\cup\id=\id\cup\kappa_n$ on
$H^1(G_F,\dbZ/n)\times H^0(G_F,\mu_n)$.
\end{proof}

See \cite{Ledet05}*{p.\ 91} and \cite{Koch02}*{Th.\ 8.13} for related results.

\subsection*{D)\ Cohomology of finite abelian $p$-groups}

For a profinite group $G$, let $H^i_\dec(G)$ be the \textbf{decomposable part} of $H^i(G)$, i.e.,
its subgroup generated by cup products of elements of $H^1(G)$.
In this subsection we show that when $G=(\dbZ/q)^n$, the group $H^2(G)$ is generated by $H^i_\dec(G)$
and the image of $\beta_G$.
In fact, for every finite abelian $p$-group $G$ of exponent divisible by $q=p^d$,
the structure of $H^*(G)$ as a graded ring was computed by Chapman (for $p\neq2$) and by
Townsley-Kulich (for $p=2$), in terms of generators and relations (\cite{Chapman82}, \cite{Townsley88}).
Since the identification of the Bockstein elements as generators  is somewhat implicit in
\cite{Chapman82} and \cite{Townsley88}, we outline an alternative proof of the required result.
It is based on the following decomposition of $H^2(G)$ to its symmetric and skew-symmetric parts,
as studied by Tignol and Amitsur (\cite{TignolAmitsur85}, \cite{Tignol86});
see also Massy \cite{Massy87}.

Let $G$ be a finite abelian group and $A$ be a finite trivial $G$-module.
Call a map $a\colon G\times G\to A$ \textbf{skew-symmetric} if it is $\dbZ$-bilinear
and $a(\sig,\sig)=0$ for all $\sig\in G$.
Then $a(\sig,\tau)=-a(\tau,\sig)$ for $\sig,\tau\in G$.
The set $\Skew(G,A)$ of all such maps forms an abelian group under addition.

For a $2$-cocycle $f\in Z^2(G,A)$ define $a_f\in \Skew(G,A)$ by $a_f(\sig,\tau)=f(\sig,\tau)-f(\tau,\sig)$.
We call $f$ \textbf{symmetric} if $a_f=0$.
Since the action of $G$ on $A$ is trivial, $2$-coboundries are symmetric.
Let $H^2(G,A)_\sym$ be the subgroup of $H^2(G,A)$ consisting of all cohomology classes of symmetric $2$-cocyles.
The map $f\mapsto a_f$ induces a homomorphism $\Psi$ with a split exact sequence \cite{TignolAmitsur85}*{Prop.\ 1.3}
\begin{equation}
\label{Tignol-Amitsur exact sequence}
0\to H^2(G,A)_\sym\to H^2(G,A)\xrightarrow{\Psi}\Skew(G,A)\to0.
\end{equation}

For $\eps$ as above,  $\eps\cup\id$ gives by \cite{Tignol86}*{Prop.\ 1.5}  an isomorphism
\begin{equation}
\label{isomorphism for H2sym}
H^1(G,\dbQ/\dbZ)\tensor_\dbZ A\xrightarrow{\sim} H^2(G,A)_\sym,
\end{equation}

Now let $A=\dbZ/q$.

\begin{prop}
\label{Psi of H2dec}
For $G=(\dbZ/q)^n$ one has $\Psi(H^2_\dec(G))=\Skew(G,\dbZ/q)$.
\end{prop}
\begin{proof}
Write $G=\langle\sig_1\rangle\times\cdots\times\langle\sig_n\rangle$ with $\sig_i$ of order $q$.
Take $\chi_1\nek \chi_n\in H^1(G)$ such that $\chi_i(\sig_i)=1$ for all $i$, and $\chi_i(\sig_k)=0$ for $i\neq k$.
For distinct $i,j$ the cohomology class $\chi_i\cup\chi_j$ is represented by the $2$-cocyle
$(\sig,\tau)\mapsto\chi_i(\sig)\chi_j(\tau)$.
Hence
\[
\begin{split}
(\Psi(\chi_i\cup\chi_j))(\sig_k,\sig_l)&=
\chi_i(\sig_k)\chi_j(\sig_l)-\chi_i(\sig_l)\chi_j(\sig_k)
\end{split}
\]
is $1$ if $(i,j)=(k,l)$, is $-1$ if $(i,j)=(l,k)$, and is $0$ otherwise.

Now given $a\in\Skew(G,\dbZ/q)$, take $\varphi=\sum_{i<j}a(\sig_i,\sig_j)\cdot\chi_i\cup\chi_j$.
For $k<l$ we get
$(\Psi(\varphi))(\sig_k,\sig_l)=a(\sig_k,\sig_l)$.
But maps in $\Skew(G,\dbZ/q)$ are determined by their values on $(\sig_k,\sig_l)$, $k<l$.
Hence  $\Psi(\varphi)=a$.
\end{proof}

\begin{prop}
\label{the Bockstein map and H2sym}
Let $G$ be a finite abelian group of exponent dividing $q$.
Then $\beta_G$ maps $H^1(G)$ isomorphically onto $H^2(G)_\sym$.
\end{prop}
\begin{proof}
As $(\dbZ/q)\tensor_\dbZ(\dbZ/q)\isom\dbZ/q$, the isomorphism (\ref{isomorphism for H2sym}) coincides with
\[
(\pi_q^*\circ\eps)\cup\id\colon H^1(G,\dbQ/\dbZ)\tensor H^0(G)\to H^2_\sym(G),
\]
where $\pi_q\colon \dbZ\to \dbZ/q$ is the natural map.
Moreover, $H^0(G)=\dbZ/q$, so
\[
H^1(G,\tfrac1q\dbZ/\dbZ)\tensor H^0(G)=H^1(G,\dbQ/\dbZ)\tensor H^0(G).
\]
By Lemma \ref{Bockstein and epsilon},  (\ref{isomorphism for H2sym}) is therefore also given by
\[
(\beta_G\circ j_q^*)\cup\id\colon H^1(G,\tfrac1q\dbZ/\dbZ)\tensor H^0(G)\to H^2_\sym(G),
\]
and the latter isomorphism may be identified with $\beta_G$.
\end{proof}

\begin{cor}
\label{elementary abelian}
Let $G=(\dbZ/q)^n$.
\begin{enumerate}
\item[(a)]
The group $H^2(G)$ is generated by $H^2_\dec(G)$ and  the image of $\beta_G$.
\item[(b)]
When $q=2$ one has $H^2(G)=H^2_\dec(G)$.
\end{enumerate}
\end{cor}
\begin{proof}
(a) follows from Proposition \ref{Psi of H2dec}, Proposition \ref{the Bockstein map and H2sym},
and the exact sequence (\ref{Tignol-Amitsur exact sequence}).
(b) follows from (a) and Lemma \ref{formula for Bockstein when q is 2}.
\end{proof}

\section{Groups of Galois relation type}
Let $G$ be again a profinite group.
The cup product $\cup\colon H^1(G)\times H^1(G)\to H^2(G)$ uniquely extends to a homomorphism
\begin{equation}
\label{cup}
\textstyle
\cup\colon H^1(G)\tensor_\dbZ H^1(G)\to H^2(G),\quad \alp\mapsto\cup\alp.
\end{equation}

\begin{defin}
\label{Galois relation type} \rm
We say that $G$ has \textbf{Galois relation type} (relative to $q$) if:
\begin{enumerate}
\item[(i)]
the kernel of the homomorphism (\ref{cup}) is generated by elements of the form $\psi\tensor\psi'$,
where $\psi,\psi'\in H^1(G)$;
\item[(ii)]
there exists $\xi\in H^1(G)$ such that for every $\psi\in H^1(G)$  one has $\psi\cup\xi+\beta_G(\psi)=0$; and
\item[(iii)]
the natural map $H^1(G)=H^1(G,\dbZ/q)\to H^1(G,\dbZ/p^i)$ is surjective for $1\leq i\leq d$ (where $q=p^d$).
\end{enumerate}
\end{defin}

As a main example, consider a field $F$ of characteristic $\neq p$ and containing a (fixed) primitive $q$th
root of unity $\zeta_q$.
Let $G_F$ be the absolute Galois group of $F$.
Let $K^M_i(F)$ be the $i$th Milnor $K$-group of $F$, and consider the Galois symbol $K^M_i(F)/q\to H^i(G_F,\dbZ/q)$.
It is an isomorphism for $i=1,2$, by the Kummer theory and the Merkurjev--Suslin theorem
(\cite{MerkurjevSuslin82}, \cite{GilleSzamuely}*{Th.\ 8.6.5}), respectively.
Moreover, it induces a commutative square
\begin{equation}
\label{Galois symbols}
\xymatrix{
(F^\times/(F^\times)^q)\tensor_\dbZ(F^\times/(F^\times)^q) \ar[r]^>>>>>{\sim}\ar[d]  &
H^1(G_F)\tensor_\dbZ H^1(G_F) \ar[d]^{\cup} \\
K^M_2(F)/q \ar[r]^{\sim}  &  H^2(G_F).  \\
}
\end{equation}
Here the left vertical map is given by
\[
\sum_{i=1}^n(a_i(F^\times)^q\tensor b_i(F^\times)^q)\mapsto\sum_{i=1}^n\{a_i,b_i\}+qK^M_2(F)
\]
and is surjective.
Its kernel is the \textbf{Steinberg group}, generated by all $a(F^\times)^q\tensor b(F^\times)^q$
with $1\in a(F^\times)^q+b(F^\times)^q$ \cite{EfratBook}*{\S24.1}.
We obtain:

\begin{prop}
\label{Galois nature}
$G=G_F$ has Galois relation type.
\end{prop}
\begin{proof}
By definition, the Steinberg group is generated by elements $a(F^\times)^q\tensor b(F^\times)^q$
which are mapped to $0$ in $K^M_2(F)/q$.
Now use the surjectivity (resp., injectivity) of the upper (resp., lower) horizontal map
in (\ref{Galois symbols}) to deduce (i).

By Proposition \ref{Bockstein and Kummer}, for $\psi\in H^1(G_F)$
one has $\beta_{G_F}(\psi)\cup\zeta_q=\psi\cup\kappa_q(\zeta_q)$, where on the left hand side we consider
$\zeta_q$ as an element of $H^0(G_F,\mu_q)$.
Identifying $\mu_q$ with $\dbZ/q$ via $\zeta_q^i\mapsto\bar i$, we get
$\beta_{G_F}(\psi)=\psi\cup\kappa_q(\zeta_q)$ in $H^2(G_F,\dbZ/q)=H^2(G_F,\mu_q)$.
Thus (ii) holds by taking $\xi=-\kappa_q(\zeta_q)$.

Finally, let $1\leq i\leq d$.
By Kummer theory, the natural epimorphism $F^\times/(F^\times)^q\to F^\times/(F^\times)^{p^i}$ yields
an epimorphism $H^1(G_F,\dbZ/q)\to H^1(G_F,\dbZ/p^i)$, proving (iii).
\end{proof}

\begin{rem}  \rm
\label{Galois nature - relative Galois groups}
Using also the surjectivity of the Galois symbol in dimension $2$, one can strengthen Proposition
\ref{Galois nature} to Galois groups $G=\Gal(E/F)$, where $E/F$ is a Galois extension,
$F$ contains a primitive $q$th root of unity, and $E$ has no proper $p$-extensions.
Indeed, then $H^1(G_E)=0$.
In (\ref{Galois symbols}) all maps are surjective.
Applying it for $E$, we obtain that $H^2(G_E)=0$ as well.
It therefore follows from the Hochschild--Serre spectral sequence that
$\inf_{G_F}\colon H^i(G)\to H^i(G_F)$ is an isomorphism for $i=1,2$ \cite{NeukirchSchmidtWingberg}*{Cor.\ 2.4.2}.
Since the cup product and the Bockstein homomorphisms commute with inflation,
conditions (i)--(iii) for $G_F$ now transform into the analogous conditions for $G$.
\end{rem}

\section{Cohomology elements of simple type}
For a profinite group $G$ we define an abelian group $\Omega(G)$ by
\[
\Omega(G)=
\begin{cases}
H^1(G)\tensor_\dbZ H^1(G), & \mathrm{if\ } q=2, \\
\bigl(H^1(G)\tensor_\dbZ H^1(G)\bigr)\oplus H^1(G), & \mathrm{if\ } q\neq2. \\
\end{cases}
\]
Define a homomorphism $\Lam_G\colon \Omega(G)\to H^2(G)$ as follows:
\[
\begin{split}
\Lam_G(\alp)=\cup\alp,\qquad\qquad &\mathrm{\quad if\ } q=2, \\
\Lam_G(\alp_1,\alp_2)=\cup\alp_1+\beta_G(\alp_2), &\mathrm{\quad if\ } q\neq2. \\
\end{split}
\]

The map $G\mapsto \Omega(G)$ is functorial.
Given an epimorphism $G_1\to G_2$ of profinite groups, the inflation map
$\inf_{G_1}\colon H^1(G_2)\to H^1(G_1)$ induces a homomorphism
$\inf_{G_1}\colon \Omega(G_2)\to \Omega(G_1)$ with a commutative square:
\begin{equation}
\label{inf and Lambda commute}
\xymatrix{
\Omega(G_2) \ar[r]^{\inf_{G_1}}  \ar[d]_{\Lam_{G_2}}  &  \Omega(G_1) \ar[d]^{\Lam_{G_1}}  \\
H^2(G_2) \ar[r]^{\inf_{G_1}}  &   H^2(G_1). \\
}
\end{equation}

\begin{lem}
\label{surjective maps}
Assume that $G$ has Galois relation type and let $\tilde G=G/G^{(2)}$.
Then $\Lam_{\tilde G}$ is surjective.
\end{lem}
\begin{proof}
For $1\leq i\leq d$, the natural map $H^1(G)\to H^1(G,\dbZ/p^i)$ is just the natural map
$\Hom(\tilde G,\dbZ/q)\to \Hom(\tilde G,\dbZ/p^i)$, so by Definition \ref{Galois relation type}(iii),
it is surjective.
Since additionally $\tilde G$ is abelian of exponent dividing $q$, it is therefore an inverse limit of
finite groups $\tilde G_j$ of the form $(\dbZ/q)^{n_j}$.
By Corollary \ref{elementary abelian}, each $H^2(\tilde G_j)$ is generated by the images of
$\cup$ and $\beta_{\tilde G_j}$ (and of $\cup$ only, if $q=2$).
Hence each $\Lam_{\tilde G_j}$ is surjective.
Conclude that $\Lam_{\tilde G}=\dirlim{\Lam_{\tilde G_j}}$ is surjective.
\end{proof}

\begin{defin}  \rm
\label{simple type}
We say that $\alp\in \Omega(G)$ has \textbf{simple type} if either:
\begin{enumerate}
\item[(i)]
$q=2$ and $\alp=\psi\tensor\psi'$ for some $\psi,\psi'\in H^1(G)$; or
\item[(ii)]
$q\neq2$ and $\alp=(\psi\tensor\psi',\psi)$ for some $\psi,\psi'\in H^1(G)$.
\end{enumerate}
\end{defin}
\noindent
In the situation of Definition \ref{simple type}, we say that $M=\Ker(\psi)\cap\Ker(\psi')$ is a \textbf{kernel} of $\alp$ (it may depend on $\psi,\psi'$).
Observe that $M$ is a normal open subgroup of $G$ and that $(\psi,\psi')$ induce an embedding of $G/M$ in $(\dbZ/q)^2$.
Hence $G^{(2)}\leq M$.
Note that inflation homomorphisms map simple type elements to simple type elements.

\begin{prop}
\label{decomposition}
Assume that $G$ has Galois relation type.
Then the group $\Ker(\Lam_G)$ is generated by elements of simple type.
\end{prop}
\begin{proof}
For $q=2$, this is just Definition \ref{Galois relation type}(i).

So suppose that $q\neq2$ and let $\alp\in\Ker(\Lam_G)$.
There exists $\psi_0\in H^1(G)$ with $\alp-(0,\psi_0)\in (H^1(G)\tensor H^1(G))\oplus\{0\}$.
Take $\xi\in H^1(G)$ as in Definition \ref{Galois relation type}(ii).
Thus $\psi_0\cup\xi+\beta_G(\psi_0)=0$, i.e.,
$\Lam_G(\psi_0\tensor\xi,\psi_0)=0$.
Let $\alp'=\alp-(\psi_0\tensor\xi,\psi_0)$.
Then $\alp'\in (H^1(G)\tensor H^1(G))\oplus\{0\}$ and $\Lam_G(\alp')=\Lam_G(\alp)=0$.
By Definition \ref{Galois relation type}(i), there exist $\psi_i,\psi'_i\in H^1(G)$, $i=1\nek n$, with
$\alp'=\sum_{i=1}^n(\psi_i\tensor\psi'_i,0)$
and $\psi_i\cup\psi'_i=0$ for all $i$.
For each $i$, we have $\Lam_G(\psi_i\tensor\xi,\psi_i)=0$.
Then
\begin{equation}
\label{decomposition to simple type elements}
\alp=(\psi_0\tensor\xi,\psi_0)+\sum_{i=1}^n(\psi_i\tensor(\psi'_i+\xi),\psi_i)-\sum_{i=1}^n(\psi_i\tensor\xi,\psi_i).
\end{equation}
Here all summands are simple type elements in $\Ker(\Lam_G)$.
\end{proof}

\begin{lem}
\label{on distinguished}
Let $\alp\in \Ker(\Lam_G)$ have simple type and kernel $M$.
Then there exist $\varphi\in H^1(M)^G$ and $\bar\alp\in \Omega(G/M)$ of simple type and with trivial kernel,
such that $\inf_G(\bar\alp)=\alp$ and $\Lam_{G/M}(\bar\alp)=\trg_{G/M}(\varphi)$.
\end{lem}
\begin{proof}
Take $\psi,\psi'$ as in Definition \ref{simple type} with $M=\Ker(\psi)\cap\Ker(\psi')$.
There exist $\bar\psi,\bar\psi'\in H^1(G/M)$ such that $\inf_G(\bar\psi)=\psi$ and  $\inf_G(\bar\psi')=\psi'$.
We define $\bar\alp\in \Omega(G/M)$  to be $\bar\psi\tensor\bar\psi'$, if $q=2$,
and $(\bar\psi\tensor\bar\psi',\bar\psi)$, if $q\neq2$.
Thus $\bar\alp$ has simple type and trivial kernel, and $\inf_G(\bar\alp)=\alp$.
By (\ref{inf and Lambda commute}) and (\ref{5-term sequence}), there is a commutative diagram with an exact row
\[
\xymatrix{
& \Omega(G/M) \ar[r]^{\inf_G}\ar[d]_{\Lam_{G/M}} & \Omega(G) \ar[d]^{\Lam_G} \\
H^1(M)^G \ar[r]^{\trg_{G/M}} & H^2(G/M) \ar[r]^{\inf_G} & H^2(G).
}
\]
It yields $\varphi\in H^1(M)^G$ as required.
\end{proof}

\begin{defin}
\label{distinguished}  \rm
Call a subgroup $N$ of $G$ \textbf{distinguished}
if there is an open subgroup $M$ of  $G$ and elements $\varphi\in H^1(M)^G$ and $\bar\alp\in \Omega(G/M)$,
with $\bar\alp$ of simple type and with trivial kernel, such that
\[
\Lam_{G/M}(\bar\alp)=\trg_{G/M}(\varphi),\quad  N=\Ker(\varphi).
\]
\end{defin}
In this case we say that $M, \varphi, \bar\alp$ are \textbf{data} for $N$.

\begin{rem} \rm
Since $\bar\alp\in \Omega(G/M)$ has trivial kernel, $G/M$ embeds in $(\dbZ/q)^2$.
Hence $(G:N)=(G:M)(M:N)|q^3$ and $G^{(2)}\leq M$.
Also, the exponent of $G/N$ divides $q^2$.
\end{rem}

\begin{exam}
\label{example of distinguished}
\rm
For every $\psi\in H^1(G)$, the subgroup $M=\Ker(\psi)$ of $G$ is distinguished.
Indeed, take $\bar\psi\in H^1(G/M)$ with $\inf_G(\bar\psi)=\psi$ and set $\bar\alp=0\in \Omega(G/M)$.
Trivially, $\bar\alp=\bar\psi\tensor0$ if $q=2$, and $\bar\alp=(0\tensor\bar\psi,0)$ if $q\neq2$.
Thus $\bar\alp$ has simple type and trivial kernel.
For $\varphi=0\in H^1(M)^G$ we have $\trg_{G/M}(\varphi)=\Lam_{G/M}(\bar\alp)=0$ and $M=\Ker(\varphi)$.
\end{exam}

\section{$G^{(3)}$ as an intersection}
Let $G$ be again a profinite group, and let $\Del_G$ be the intersection of all distinguished subgroups of $G$.

\begin{prop}
\label{remarks on Del}
$G^{(3)}\leq\Del_G\leq G^{(2)}$.
\end{prop}
\begin{proof}
Let $N$ be a distinguished subgroup of $G$.
Thus there exists an open normal subgroup $M$ of $G$ and $\varphi\in H^1(M)^G$ such that
$\Ker(\varphi)=N$ and $G^{(2)}\leq M$.
Hence Lemma \ref{Pontryagin} gives
\[
G^{(3)}=(G^{(2)})^q[G^{(2)},G]\leq M^q[M,G]\leq\Ker(\varphi)=N.
\]
Consequently, $G^{(3)}\leq \Del_G$.

By Lemma \ref{Pontryagin} again, $\bigcap_{\psi\in H^1(G)}\Ker(\psi)=G^{(2)}$.
Since each $\Ker(\psi)$ is distinguished  (Example \ref{example of distinguished}),
we get that $\Del_G\leq G^{(2)}$.
\end{proof}

\begin{thm}
\label{formal main theorem}
If $G$ has Galois relation type, then
$G^{(3)}=\Del_G$.
\end{thm}
\begin{proof}
By Proposition \ref{remarks on Del}, $G^{(3)}\leq \Del_G$.

For the converse inclusion, let $\tilde G=G/G^{(2)}$.
It follows from Lemma \ref{Pontryagin} (with $N=G$) that the map $\res_{G^{(2)}}\colon H^1(G)\to H^1(G^{(2)})$ is trivial.
Hence, by (\ref{5-term sequence}), $\inf_G\colon H^1(\tilde G)\to H^1(G)$ is an isomorphism.
Consequently, $\inf_G\colon \Omega(\tilde G)\to \Omega(G)$ is also an isomorphism.

Now let $\varphi\in H^1(G^{(2)})^G$.
By Lemma \ref{surjective maps}, $\Lam_{\tilde G}$ is surjective, so there exists $\tilde\alp\in \Omega(\tilde G)$
with $\trg_{\tilde G}(\varphi)=\Lam_{\tilde G}(\tilde\alp)$.
By (\ref{inf and Lambda commute}) and (\ref{5-term sequence}),
\[
\Lam_G(\textstyle{\inf_G}(\tilde\alp))=\textstyle{\inf_G}(\Lam_{\tilde G}(\tilde\alp))=\inf_G(\trg_{\tilde G}(\varphi))=0.
\]
By Proposition \ref{decomposition} we may therefore write $\inf_G(\tilde\alp)=\sum_{i=1}^n\alp_i$,
where $\alp_1\nek\alp_n\in \Ker(\Lam_G)$ have simple type.

For each $0\leq i\leq n$ let $M_i$ be a kernel for $\alp_i$.
Recall that $G^{(2)}\leq M_i$, so (\ref{inf and Lambda commute}) again gives a commutative diagram
\begin{equation}
\label{cd}
\xymatrix{
\Omega(G/M_i) \ar[r]^{\inf_{\tilde G}}\ar[d]_{\Lam_{G/M_i}} & \Omega(\tilde G) \ar[r]^{\inf_G}\ar[d]^{\Lam_{\tilde G}} & \Omega(G)\ar[d]^{\Lam_G} \\
H^2(G/M_i) \ar[r]^{\inf_{\tilde G}} & H^2(\tilde G) \ar[r]^{\inf_G} & H^2(G).
}
\end{equation}
Lemma \ref{on distinguished} gives rise to $\bar\alp_i\in \Omega(G/M_i)$ of simple type and with
trivial kernel and to $\varphi_i\in H^1(M_i)^G$ such that $\inf_G(\bar\alp_i)=\alp_i$ and
$\Lam_{G/M_i}(\bar\alp_i)=\trg_{G/M_i}(\varphi_i)$.
In particular, $\Ker(\varphi_i)$ is distinguished.
For each $i$ let $\tilde\alp_i=\inf_{\tilde G}(\bar\alp_i)$.
It also has simple type, and one has $\alp_i=\inf_G(\tilde\alp_i)$.
By (\ref{cd}), $\inf_G(\Lam_{\tilde G}(\tilde\alp_i))=0$.
Moreover,
\[
\textstyle{\inf_G}(\tilde\alp)=\sum_{i=1}^n\alp_i=\sum_{i=1}^n\textstyle{\inf_G}(\tilde\alp_i).
\]
But $\inf_G\colon \Omega(\tilde G)\to \Omega(G)$ is an isomorphism, so $\tilde\alp=\sum_{i=1}^n\tilde\alp_i$.

Next (\ref{5-term sequence}) and (\ref{trg res and inf commute}) give a commutative diagram with an exact row:
\begin{equation*}
\xymatrix{
  & H^1(M_i)^G\ar[r]^{\trg_{G/M_i}} \ar[d]_{\res_{G^{(2)}}} & H^2(G/M_i)\ar[d]^{\inf_{\tilde G}}  \\
 0\ar[r] & H^1(G^{(2)})^G \ar[r]^{\trg_{\tilde G}} & H^2(\tilde G). \\
 }
\end{equation*}
Using this and (\ref{cd}) we compute:
\[
\begin{split}
\trg_{\tilde G}(\varphi)
=\Lam_{\tilde G}(\tilde\alp)
&=\sum_{i=1}^n\Lam_{\tilde G}(\tilde\alp_i)=\sum_{i=1}^n\Lam_{\tilde G}(\textstyle{\inf_{\tilde G}}(\bar\alp_i))=\sum_{i=1}^n{\textstyle{\inf_{\tilde G}}}(\Lam_{G/M_i}(\bar\alp_i)) \\
&=\sum_{i=1}^n(\textstyle{\inf_{\tilde G}}\circ\trg_{G/M_i})(\varphi_i)=\sum_{i=1}^n(\trg_{\tilde G}\circ\res_{G^{(2)}})(\varphi_i).\\
\end{split}
\]
Since  $\trg_{\tilde G}$ is injective, $\varphi=\sum_{i=1}^n\res_{G^{(2)}}(\varphi_i)$, so by Proposition \ref{remarks on Del},
\[
\Ker(\varphi)\geq\bigcap_{i=1}^n\Ker(\res_{G^{(2)}}(\varphi_i))=G^{(2)}\cap\bigcap_{i=1}^n\Ker(\varphi_i)
\geq G^{(2)}\cap \Del_G=\Del_G.
\]
Since $\varphi\in H^1(G^{(2)})^G$ was arbitrary, we deduce from Lemma \ref{Pontryagin} that
\[
G^{(3)}=(G^{(2)})^q[G^{(2)},G]=\bigcap_{\varphi\in H^1(G^{(2)})^G}\Ker(\varphi)\geq\Del_G.
\qedhere
\]
\end{proof}

\begin{cor}
Let $G$ be a profinite group of Galois relation type.
Then $G^{(3)}$ is an intersection of normal open subgroups $N$ of $G$ with $G/N$ of order dividing $q^3$
and exponent dividing $q^2$.
\end{cor}

\section{Extensions}
Let $\bar G$ be a finite group and $A$ a finite trivial $\bar G$-module.
We consider central extensions
\[
\omega:\qquad 0\ \to\ A\ \xrightarrow{f} B\ \xrightarrow{g}\ \bar G\ \to\ 1.
\]
For a group isomorphism $\theta\colon\bar G'\to\bar G$ define an extension
\[
\omega^{\theta}:\quad 0\ \to\ A\ \xrightarrow{f} B\ \xrightarrow{\theta\inv\circ g}\ \bar G'\ \to\ 1.
\]
When there is a commutative diagram of central extensions
\[
\xymatrix{
\omega: & 0\ar[r] & A\ar[r]^{f}\ar@{=}[d] & B\ar[r]^{g}\ar[d]^{h}_{\wr} & \bar G\ar[r]\ar@{=}[d] & 1 \\
\omega': & 0\ar[r] & A'\ar[r]^{f} & B'\ar[r]^{g'} & \bar G\ar[r] & 1, \\
}
\]
with $h$ an isomorphism, $\omega$ and $\omega'$ are called \textbf{equivalent}.
Let $\Ext(\bar G,A)$ be the set of all equivalence classes $[\omega]$ of extensions $\omega$ as above.

The \textbf{Baer sum} \cite{CartanEilenberg56}*{Ch.\ XIV, \S1}  of central extensions
\[
\omega_i:\qquad 0\ \to\ A\xrightarrow{f_i}\ B_i\ \xrightarrow{g_i}\ \bar G\ \to\ 1, \quad i=1,2,
\]
is the central extension
\[
0\ \to\ A\xrightarrow{\overline{(f_1,1)}=\overline{(1,f_2)}}\ B\ \xrightarrow{g_1=g_2}\ \bar G\ \to\ 1.
\]
where for the fibred product $B_1\times_{\bar G} B_2$ we set
\[
B=(B_1\times_{\bar G} B_2)/\{(f_1(a),f_2(a)\inv)\ |\ a\in A\}.
\]
This induces an abelian group structure on $\Ext(\bar G,A)$, which is functorial in $\bar G$
(contravariantly) and in $A$ (covariantly).

There is a canonical isomorphism  $\Ext(\bar G,A)\isom H^2(\bar G,A)$ which is
functorial in both $\bar G$ and $A$ \cite{NeukirchSchmidtWingberg}*{Th.\ 1.2.4}.
Specifically, the cohomology class of an inhomogeneous normalized $2$-cocycle $\alp\colon \bar G^2\to A$
corresponds to the class of $[\omega]$, where $B=A\times \bar G$ as sets, and the group law is given for $a,b\in A$ and $\sig,\tau\in \bar G$ by
\begin{equation}
\label{group law}
(a,\sig)*(b,\tau)=(a+b+\alp(\sig,\tau),\sig\tau).
\end{equation}
Conversely, given $\omega$ as above, choose a set-theoretic section $s\colon \bar G\to B$ of $g$ with $s(1)=1$.
The map $\alp\colon \bar G\times \bar G\to A$, given by
$\alp(\bar \sig_1,\bar\sig_2)=s(\bar\sig_1)s(\bar\sig_2)s(\bar\sig_1\bar \sig_2)\inv$,
is an inhomogenous normalized $2$-cocyle whose cohomology class corresponds to $[\omega]$.

\begin{rem}[\cite{GilleSzamuely}*{Remark 3.3.11}, \cite{Ledet05}*{p.\ 33}] \rm
Let $\bar G\to\tilde G$ be an epimorphism and let $A$ be a $\tilde G$-module, whence a $\bar G$-module in the natural way.
Then $\inf_{\bar G}\colon H^2(\tilde G,A)\to H^2(\bar G,A)$ corresponds
to the map $\inf_{\bar G}\colon \Ext(\tilde G,A)\to \Ext(\bar G,A)$ sending the class of
\[
\tilde\omega:\qquad 0\ \to\  A\ \xrightarrow{f}\  B\ \xrightarrow{g}\ \tilde G\ \to\  1
\]
to the class of
\[
{\textstyle \inf_{\bar G}}(\tilde\omega):\quad  0\to A\xrightarrow{(f,1)}
B\times_{\tilde G}\bar G\xrightarrow{(b,\bar\sig)\mapsto \bar\sig} \bar G\to 1.
\]
\end{rem}

In particular we have:

\begin{lem}
\label{inflation to a direct product}
Suppose that $\bar G=\tilde G\times\tilde G'$ and let $\tilde\omega$ be as above.
Then $\inf_{\bar G}(\tilde\omega)$ is equivalent to
\[
0 \to\ A\ \xrightarrow{(f,1)} \ B\times \tilde G'\ \xrightarrow{g\times\id}\ \tilde G\times\tilde G'\ \to\ 1.
\]
\end{lem}
\begin{proof}
Use the commutative triangle
\[
\xymatrix{
B\times_{\tilde G}\bar G \ar[rr]^{(b,(\tilde\sig,\tilde\sig'))\mapsto(b,\tilde\sig')}
\ar[rd]_{(b,(\tilde\sig,\tilde\sig'))\mapsto(\tilde\sig,\tilde\sig')} & & B\times \tilde G' \ar[ld]^{g\times\id} \\
& \bar G
}
\]
where the horizontal map is an isomorphism.
\end{proof}

\section{Embedding problems}
Let $G$ be a profinite group.
The following proposition is due to Hoechsmann \cite{Hoechsmann68}*{2.1}:

\begin{prop}
\label{Hoechsmann}
Let $M$ be an open normal subgroup of $G$.
Consider the embedding problem
\begin{equation}
\label{embedding problem}
\xymatrix{
&&&&   G  \ar[dl]_{\Phi} \ar[d] \\
\omega: &0\ar[r] &  A\ar[r] &   B\ar[r]  & G/M\ar[r] & 1
}
\end{equation}
where $A$ is a finite $G/M$-module, and let $\alp\in H^2(G/M,A)$ be the cohomology class corresponding to $[\omega]$.
Then the restriction map $\Phi\mapsto\varphi=\Phi|_M$ is a surjection from
\begin{enumerate}
\item[(a)]
the continuous homomorphisms $\Phi\colon G\to B$ making (\ref{embedding problem}) commutative; to
\item[(b)]
the elements $\varphi$ of $H^1(M,A)^G$ with $\trg_{G/M}(\varphi)=\alp$.
\end{enumerate}
In particular, there exists $\Phi$ as in (a) if and only if $\inf_G(\alp)=0$.
\end{prop}
Here, the last sentence follows from the bijection using (\ref{5-term sequence}).

Now suppose that $A=\dbZ/q$ with the trivial $G$-action, where $q=p^d$.

\begin{prop}
\label{distinguished and embedding problems}
Let $M$ be an open normal subgroup of $G$ and let $\bar\alp\in \Omega(G/M)$ have simple type and trivial kernel.
The following conditions on an open subgroup $N$ of $M$ are equivalent:
\begin{enumerate}
\item[(a)]
$N$ is a distinguished subgroup of $G$ with data $M,\bar\alp$;
\item[(b)]
$N$ is normal in $G$ and there is a commutative diagram
\[
\xymatrix{
&&&& G/N \ar[d]\ar[dl]_{h} \\
\omega: & 0\ar[r] & \dbZ/q \ar[r] & B\ar[r] & G/M\ar[r] & 1,
}
\]
where $\omega$ is an extension corresponding to $\Lam_{G/M}(\bar\alp)$, the vertical
map is the natural projection, and  $h$ is a monomorphism.
\end{enumerate}
\end{prop}
\begin{proof}
(a)$\Rightarrow$(b):\quad
By assumption, there exists $\varphi\in H^1(M)^G$ such that
$\Lam_{G/M}(\bar\alp)=\trg_{G/M}(\varphi)\in H^2(G/M)$ and $N=\Ker(\varphi)$.
In particular, $N$ is normal in $G$.
Choose a central extension $\omega$ as above corresponding to $\Lam_{G/M}(\bar\alp)$.
Proposition \ref{Hoechsmann} yields a continuous homomorphism $\Phi\colon G\to B$ such that
(\ref{embedding problem}) commutes and $\varphi=\Phi|_M$.
Then $N=\Ker(\varphi)=M\cap\Ker(\Phi)$.
Consequently, $\Phi$ induces  a homomorphism $h\colon G/N\to B$ whose restriction to $M/N$ is injective.
It follows that $h$ is also injective.

\medskip

(b)$\Rightarrow$(a):\quad
Lift $h$ to a homomorphism $\Phi\colon G\to B$ with kernel $N$.
Then (\ref{embedding problem}) commutes.
Then $\varphi=\Phi|_M\in H^1(M)^G$.
By Proposition \ref{Hoechsmann}, $\Lam_{G/M}(\bar\alp)=\trg_{G/M}(\varphi)$
and $N=M\cap\Ker(\Phi)=\Ker(\varphi)$, giving (a).
\end{proof}

\section{Special extensions}
Proposition \ref{distinguished and embedding problems} allows an explicit determination of the
distinguished subgroups $N$ of a profinite group $G$ by means of the quotients $G/N$.
We now carry out this computation for $q=p$ prime, based on an analysis of several
central extensions of small $p$-groups.
We first recall the structure of the nonabelian groups of order $p^3$.
When $p=2$ these are:
\begin{itemize}
\item
the \textbf{dihedral group} of order $8$,
\[
D_4=\langle r,s\ |\ r^4=s^2=(rs)^2=1\rangle;
\]
\item
the \textbf{quaternionic group}
\[
Q_8=\langle r,s\ |\ r^4=1,\ [r,s]=r^2=s^2 \rangle.
\]
\end{itemize}
For $p$ odd, there are two isomorphism types of groups of order $p^3$:
\begin{itemize}
\item
the \textbf{Heisenberg group} of order $p^3$ and exponent $p$,
\[
H_{p^3}=\langle r,s,t\ |\ r^p=s^p=t^p=1,\ [r,t]=[s,t]=1,\ [r,s]=t \rangle;
\]
\item
the \textbf{extra-special group} of order $p^3$ and exponent $p^2$,
\[
M_{p^3}=\langle r,s\ |\ r^{p^2}=s^p=1,\  [r,s]=r^p\rangle. \
\]
\end{itemize}

\begin{rems}
\label{comments on Mp3 D8}
\rm
(a)\quad
When $p=2$ one has $M_8=D_4$.
However,  we will keep in this case the traditional notation $D_4$,
and write $M_{p^3}$ only when $p\neq2$.

\medskip

(b)\quad
Let $G$ be one of the groups $D_4$, $Q_8$, when $p=2$, or $H_{p^3}$, $M_{p^3}$, when $p\neq2$.
Then the unique normal subgroup of $G$ of order $p$ is its center, which coincides with the
Frattini subgroup $G^{(2)}$ \cite{MassyNguyenQuangDo77}*{\S3.1}.
Therefore $G^{(3)}=(G^{(2)})^p[G^{(2)},G]=1$.

\medskip

(c)\quad
In $M_{p^3}$ (for $p\neq2$) one has $r^js^i=s^ir^{(1+ip)j}$ for all $i,j\geq0$.
In particular,  $[s,r^p]=1$.
Further, by induction,
$(s^ir^j)^k=s^{ki}r^{(1+(k-1)ip/2)kj}$ for $k\geq0$.
It follows that $(s^ir^j)^p=1$ if and only if $p|j$.
\end{rems}

We define epimorphisms from these groups onto $(\dbZ/p)^2$ as follows:
\begin{align*}
&\rho\colon D_4\to(\dbZ/2)^2, \qquad\  r\mapsto(\bar1,\bar1),\ s\mapsto(\bar0,\bar1); \\
&\lam\colon H_{p^3}\to (\dbZ/p)^2, \qquad r\mapsto(\bar1,\bar0),\ s\mapsto(\bar0,\bar1), \ t\mapsto(\bar0,\bar0); \\
&\lam'\colon M_{p^3}\to (\dbZ/p)^2, \quad\ \, r\mapsto(\bar1,\bar0), \ s\mapsto(\bar0,\bar1).
\end{align*}

\begin{rem}
\label{subgroups of Mp3}
\rm
For later use we note that no proper subgroup of $D_4$ (resp., $M_{p^3}$) is mapped surjectively by $\rho$
(resp., $\lam'$).
\end{rem}

The following central extensions will be needed in the sequel:
\begin{align*}
&\omega_0: \qquad
0\ \to\  \dbZ/p\ \xrightarrow{\id}\ \dbZ/p\ \to\ 0\to\ 0;&   \\
&\omega_1: \qquad
0\ \to\ \dbZ/p\ \xrightarrow{\bar i\mapsto(\bar i,\bar0)}\ (\dbZ/p)^2\ \xrightarrow{(\bar i,\bar j)
\mapsto\bar j}\ \dbZ/p\ \to\ 0; \\
&\omega_2: \qquad
0\ \to \ \dbZ/p\ \xrightarrow{\bar i\mapsto \overline{pi}}\ \dbZ/p^2\ \xrightarrow{\bar i\mapsto\bar i}
\ \dbZ/p\ \to\ 0; & \\
& \omega_3: \qquad
0\ \to \ \dbZ/2\ \xrightarrow{\bar i\mapsto r^{2i}}\ D_4\ \xrightarrow{\rho} \ (\dbZ/2)^2\ \to\ 0; & \\
&\omega_4: \qquad
0\ \to\ \dbZ/p\ \xrightarrow{\bar i\mapsto t^i}\ H_{p^3}\ \xrightarrow{\lam}\ (\dbZ/p)^2\ \to\ 0 \qquad\ \ (p\neq2); &\\
&\omega_5: \qquad
0\ \to\ \dbZ/p\ \xrightarrow{\bar i\mapsto r^{pi}}\ M_{p^3}\ \xrightarrow{\lam'}\ (\dbZ/p)^2\ \to\ 0 \qquad (p\neq2); &\\
&\omega_6: \qquad
0\ \to\ \dbZ/p\ \xrightarrow{\bar i\mapsto (\overline{pi},0)}\ (\dbZ/p^2)\oplus(\dbZ/p)
\ \xrightarrow{(\bar i,\bar j)\mapsto(\bar i,\bar j)}\ (\dbZ/p)^2\ \to\ 0. & \\
\end{align*}
Thus $[\omega_0]$, $[\omega_1]$ are the trivial classes of $\Ext(0,\dbZ/p)$, $\Ext(\dbZ/p,\dbZ/p)$, respectively,
and $[\omega_1]$ is the inflation of $[\omega_0]$.
Likewise
\begin{equation}
\label{inflation of extensions}
\textstyle{\inf_{(\dbZ/p)^2}}([\omega_2])=[\omega_6]
\end{equation}
relative to the projection $\pr_1\colon(\dbZ/p)^2\to\dbZ/p$ on the first coordinate.

\begin{lem}
\label{Baer example}
For $p\neq2$, the Baer sum of $[\omega_4]$ and $[\omega_6]$ is $[\omega_5]$.
\end{lem}
\begin{proof}
Let
\[
\tilde B=\langle \tilde r,\tilde s,\tilde t\ |\ \tilde r^{p^2}=\tilde s^p=\tilde t^p=1,
\ [\tilde r,\tilde t]=[\tilde s,\tilde t]=1,\ [\tilde r,\tilde s]=\tilde t \rangle.
\]
There is a commutative square
\[
\xymatrix{
\tilde B\ar[d]_{f}\ar[r]^{\qquad j\qquad} & H_{p^3}\ar[d]^{\lam} \\
\quad(\dbZ/p^2)\oplus(\dbZ/p)\quad \ar[r]^{\qquad(\bar i,\bar j)\mapsto(\bar i,\bar j)} & \quad(\dbZ/p)^2\quad,
}
\]
where $j$ maps $\tilde r,\tilde s,\tilde t$ to $r,s,t$, respectively,
and $f$ maps $\tilde r,\tilde s,\tilde t$ to $(\bar1,\bar0)$, $(\bar0,\bar1)$, $(\bar0,\bar0)$, respectively.
Moreover, this square is cartesian, i.e.,
\[
(j,f)\colon\tilde B\to H_{p^3}\times_{(\dbZ/p)^2}((\dbZ/p^2)\oplus(\dbZ/p))
\]
is an isomorphism.
The Baer sum is therefore the equivalence class of
\[
0\ \to\ \dbZ/p\ \xrightarrow{\bar i\mapsto \tilde t^i} \ B=
\tilde B/\langle\tilde t^i\cdot\tilde r^{-pi}\\ |\ \bar i\in\dbZ/p\rangle \ \xrightarrow{\lam\circ j}\ (\dbZ/p)^2\ \to\ 0.
\]
Hence $B$ is obtained from $\tilde B$ by adding the relation $\tilde t=\tilde r^p$.
Using Remark \ref{comments on Mp3 D8}(c) we deduce that
\[
B\isom\langle r,s\ |\ r^{p^2}=s^p=1,\ [s,r^p]=1,\ [r,s]=r^p\rangle=M_{p^3},
\]
and the Baer sum is $[\omega_5]$.
\end{proof}

\section{Extensions and simple type elements}

We assume again that $q=p$ is prime.
Let $\bar G$ be a finite group.
In this section we compute the extensions corresponding to $\Lam_{\bar G}(\bar\alp)$ for
$\bar\alp\in \Omega(\bar G)$ of simple type and trivial kernel.
Some of these facts are quite well-known in a Galois setting, as is systematically described in
Ledet's book \cite{Ledet05} (see also \cite{Frohlich85}*{7.7}, \cite{MassyNguyenQuangDo77}),
but we derive them in a more abstract group-theoretic setting.

\subsection*{A)\ Cup products}
Let $\bar \psi,\bar\psi'\in H^1(\bar G)$.
We compute the extensions corresponding to $\bar\psi\cup\bar\psi'\in H^2(\bar G)$ in various situations.
For $\bar\psi_1\nek\bar\psi_k\in H^1(\bar G)$ such that $(\bar\psi_1\nek\bar\psi_k)\colon \bar G\to(\dbZ/p)^k$ is an isomorphism
and for a central extension $\omega$ of $\bar G$ by $\dbZ/p$,
we write $\omega^{(\bar\psi_1\nek\bar\psi_k)}$ for the modified extension of $(\dbZ/p)^k$ by $\dbZ/p$
as in the beginning of \S6.

\begin{prop}
\label{extensions of cup}
Suppose that  $\Ker(\bar\psi)\cap\Ker(\bar\psi')=1$.
\begin{enumerate}
\item[(a)]
If $\bar\psi=\bar\psi'=0$, then $\bar\psi\cup\bar\psi'$ corresponds to $\omega_0$.
\item[(b)]
If  $\bar\psi\neq0$, $\bar\psi'=0$ (resp., $\bar\psi=0$, $\bar\psi'\neq0$),
then $\bar\psi\cup\bar\psi'$ corresponds to $\omega_1^{\bar\psi}$ (resp., $\omega_1^{\bar\psi'}$).
\item[(c)]
If $p=2$ and $\bar\psi=\bar\psi'\neq0$, then $\bar\psi\cup\bar\psi'$ corresponds to $\omega_2^{\bar\psi}$.
\item[(d)]
If $p\neq2$ and $\bar\psi,\bar\psi'\neq0$ are $\dbF_p$-linearly dependent,
then $\bar\psi\cup\bar\psi'$ corresponds to $\omega_1^{\bar\psi}$.
\item[(e)]
If $p=2$ and $\bar\psi,\bar\psi'$ are $\dbF_p$-linearly independent,
then $\bar\psi\cup\bar\psi'$ corresponds to $\omega_3^{(\bar\psi,\bar\psi')}$.
\item[(f)]
If $p\neq2$ and $\bar\psi,\bar\psi'$ are $\dbF_p$-linearly independent,
then $\bar\psi\cup\bar\psi'$ corresponds to $\omega_4^{(\bar\psi,\bar\psi')}$.
\end{enumerate}
\end{prop}
\begin{proof}
Consider the central extension
\[
\omega:\qquad  0\ \to\ \dbZ/p\ \xrightarrow{f}\ B\ \xrightarrow{g}\ \bar G\ \to\ 1
\]
corresponding to $\bar\psi\cup\bar\psi'$.
An inhomogeneous normalized $2$-cocyle $\bar G\times \bar G\to \dbZ/p$ representing
$\bar\psi\cup\bar\psi'$ is given by $(\sig,\tau)\mapsto\bar\psi(\sig)\cdot\bar\psi'(\tau)$.
Therefore $B=(\dbZ/p)\times \bar G$, with the group law
\begin{equation}
\label{group law for cup}
(a,\sig)*(b,\tau)=(a+b+\bar\psi(\sig)\bar\psi'(\tau),\sig\tau)
\end{equation}
for $a,b\in\dbZ/p$ and $\sig,\tau\in \bar G$ (see (\ref{group law})).
The trivial element of $B$ is $(0,1)$, and one has $f(a)=(a,1)$ and
$g(a,\sig)=\sig$ for $a\in \dbZ/p$ and $\sig\in\bar G$.
By induction,
\[
(a,\sig)^i=(ia+\tfrac{i(i-1)}2\bar\psi(\sig)\bar\psi'(\sig), \sig^i), \quad i=0,1,2,\ldots\ .
\]
We examine the various possibilities.

\medskip

(a)\quad
Immediate.

\medskip

(b)\quad
Here $\bar\psi$ (resp., $\bar\psi'$) is an isomorphism $\bar G\to\dbZ/p$
and $B$ is just the direct product $(\dbZ/p)\times\bar G$.
The assertion follows.

\medskip

(c)\quad
The assumptions imply that $\bar\psi=\bar\psi'\colon \bar G\to\dbZ/2$ is an isomorphism.
Let $\sig_0$ be the generator of $\bar G$.
Then $(0,\sig_0)^2=(1,1)$ and $(0,\sig_0)^4=(0,1)$ in $B$.
Hence $B\isom\dbZ/4$ and $\omega$ is equivalent to $\omega_2^{\bar\psi}$.

\medskip

(d)\quad
Here $\bar\psi\colon \bar G\to\dbZ/p$ is an isomorphism.
Since $p\neq2$ and $\cup$ is alternate, $\bar\psi\cup\bar\psi'=0$.
Hence $\omega$, and therefore also $\omega^{\bar\psi}$ split, so $B\isom(\dbZ/p)^2$.
Moreover, pick $b\in B$ such that  $(\bar\psi\circ g)(b)=\bar 1$.
Then the map $B\to(\dbZ/p)^2$, $f(\bar1)\mapsto (\bar1,\bar0)$, $b\mapsto (\bar0,\bar1)$,
is an isomorphism making the following diagram commutative:
\[
\xymatrix{
\omega: & 0\ar[r] & \dbZ/p\ar@{^{(}->}[r]^{f\qquad }\ar@{=}[d] & B=(\dbZ/p)\times\bar G
\ar[r]^{\qquad\quad g}\ar[d]_{\wr} & \bar G\ar[r]\ar[d]_{\wr}^{\bar\psi} & 1 \\
\omega_1: & 0\ar[r] & \dbZ/p\ar@{^{(}->}[r]^{\bar i\mapsto(\bar i,\bar 0)} & (\dbZ/p)^2
\ar[r]^{ (\bar i,\bar j)\mapsto\bar j} & \dbZ/p \ar[r] &  \ 0.
}
\]
Thus $\omega$ is equivalent to $\omega_1^{\bar\psi}$.

\medskip

(e), (f)\quad
Here $(\bar\psi,\bar\psi')\colon \bar G\to(\dbZ/p)^2$ is an isomorphism.
Take $\sig_1,\sig_2\in \bar G$ with
\[
\bar\psi(\sig_1)=1,\ \bar\psi(\sig_2)=0,\ \bar\psi'(\sig_1)=0,\ \bar\psi'(\sig_2)=1.
\]

When $p=2$, we set  $\tilde r=(1,\sig_1\sig_2)$, $\tilde s=(0,\sig_2)$ and compute in $B$:
\[
\tilde r^2=(1,1) , \ \tilde r^4=(0,1), \ \tilde s^2=(0,1) , \ \tilde r\tilde s=(0,\sig_1),
\ (\tilde r\tilde s)^2=(0,1).
\]
We get an isomorphism $B\isom D_4$, $\tilde r\mapsto r$, $\tilde s\mapsto s$, and a diagram
\[
\xymatrix{
\omega: &0\ar[r] & \dbZ/2\ar@{^{(}->}[r]^{f\qquad}\ar@{=}[d] &B=(\dbZ/2)\times\bar G
\ar[r]^{\qquad\quad g}\ar[d]_{\wr}  & \bar G\ar[d]^{(\bar\psi,\bar\psi')}_{\wr} \ar[r]& 1 \\
\omega_3:& 1\ar[r] & \dbZ/2\ar@{^{(}->}[r]^{\bar i\mapsto r^{2i}}  & D_4\ar[r]^{\rho}  & (\dbZ/2)^2  \ar[r]& 0\\
}
\]
which is commutative with exact  rows.
Hence $\omega$ is equivalent to $\omega_3^{(\bar\psi,\bar\psi')}$.

For $p$ odd, $B$ has exponent $p$.
Set $\tilde r=(0,\sig_1)$, $\tilde s=(0,\sig_2)$, $\tilde t=(1,1)$.
Then
\[
\tilde r\tilde t=\tilde t\tilde r=(1,\sig_1), \quad \tilde s\tilde t=\tilde t\tilde s=(1,\sig_2),
\quad \tilde r\tilde s=\tilde t\tilde s\tilde r=(1,\sig_1\sig_2).
\]
This gives an isomorphism $B\isom H_{p^3}$, $\tilde r\mapsto r$, $\tilde s\mapsto s$, $\tilde t\mapsto t$,
and a commutative diagram
\[
\xymatrix{
\omega: & 0\ar[r] & \dbZ/p\ar@{^{(}->}[r]^{f\qquad}\ar@{=}[d] &B=(\dbZ/p)\times\bar G
\ar[r]^{\qquad\quad g}\ar[d]_{\wr}  & \bar G\ar[d]^{(\bar\psi,\bar\psi')}_{\wr} \ar[r]& 1 \\
\omega_4: & 0\ar[r] & \dbZ/p\ar@{^{(}->}[r]^{\bar i\mapsto t^i}  & H_{p^3}\ar[r]^{\lambda}  & (\dbZ/p)^2  \ar[r]& 0 .\\
}
\]
Therefore $\omega$ is equivalent in this case to $\omega_4^{(\bar\psi,\bar\psi')}$.
\end{proof}

\subsection*{B)\ Bockstein elements}

\begin{prop}
\label{extension of Bockstein}
If $0\neq \bar\psi\in H^1(\bar G)$ and $\bar G\isom\dbZ/p$,  then $\beta_{\bar G}(\bar\psi)$
corresponds to $\omega_2^{\bar\psi}$.
\end{prop}
\begin{proof}
As a connecting homomorphism in a cohomology exact sequence, $\beta_{\bar G}\colon H^1(\bar G)\to H^2(\bar G)$
is defined as follows  \cite{NeukirchSchmidtWingberg}*{Ch.\ I, \S3}:
let $\pr\colon\dbZ/p^2\to\dbZ/p$ be the natural projection.
Given a nonzero $\bar\psi\in H^1(\bar G)$, we consider it as an inhomogeneous $1$-cocycle,
and lift it to a map $\hat \psi\colon \bar G\to\dbZ/p^2$ with $\bar\psi=\pr\circ\hat\psi$.
Then the map
\[
\chi\colon \bar G\times\bar G\to\dbZ/p, \qquad \chi(\sig_1,\sig_2)
=\hat\psi(\sig_1)+\hat\psi(\sig_2)-\hat\psi(\sig_1\sig_2)
\]
is a normalized $2$-cocycle with cohomology class $\beta_{\bar G}(\bar\psi)$.

On the other hand, $\bar\psi\colon \bar G\to\dbZ/p$ is an isomorphism, so $\hat\psi$
is a section of the epimorphism $\bar\psi\inv\circ\pr\colon\dbZ/p^2\to\bar G$.
By the remarks in \S6, the cohomology class $\beta_{\bar G}(\bar\psi)$ of $\chi$ therefore corresponds to the extension
\[
\omega_2^{\bar\psi}:\qquad  0\ \to\ \dbZ/p\ \to\ \dbZ/p^2\ \xrightarrow{\bar\psi\inv\circ\pr}\  \bar G\ \to\  1.
\qedhere
\]
\end{proof}

\begin{cor}
\label{extensions for linearly dependent elements}
Suppose that $p\neq2$ and let $\bar\psi,\bar\psi'\in H^1(\bar G)$ be $\dbF_p$-linearly dependent, $\bar\psi\neq0$.
Then $\Lam_{\bar G}(\bar\psi\tensor\bar\psi',\bar\psi)$ corresponds to $\omega_2^{\bar\psi}$.
\end{cor}
\begin{proof}
Since $\cup$ is alternate, $\bar\psi\cup\bar\psi'=0$.
Hence $\Lam_{\bar G}(\bar\psi\tensor\bar\psi',\bar\psi)=\beta_{\bar G}(\bar\psi)$.
Now use Proposition \ref{extension of Bockstein}.
\end{proof}

When $\bar\psi,\bar\psi'$ are $\dbF_p$-linearly independent the computation is more involved.
It is sufficient for us to consider only the case $p\neq2$.

\begin{prop}
\label{extensions for linearly independent elements}
Suppose that $p\neq2$.
Let $\bar\psi,\bar\psi'\in H^1(\bar G)$ be $\dbF_p$-linearly independent with $\Ker(\bar\psi)\cap\Ker(\bar\psi')=1$.
Then $\Lam_{\bar G}(\bar\psi\tensor\bar\psi',\bar\psi)$ corresponds to $\omega_5^{(\bar\psi,\bar\psi')}$.
\end{prop}
\begin{proof}
We may decompose $\bar G=\tilde G\times\tilde G'$, where $\tilde G,\tilde G'\isom\dbZ/p$ and there exists
$\tilde\psi\in H^1(\tilde G)$ with $\inf_{\bar G}(\tilde\psi)=\bar\psi$.
By Proposition \ref{extension of Bockstein}, $\beta_{\tilde G}(\tilde\psi)$ corresponds to $\omega_2^{\tilde\psi}$.
Hence $\beta_{\bar G}(\bar\psi)=\inf_{\bar G}(\beta_{\tilde G}(\tilde\psi))$ corresponds to
$\inf_{\bar G}(\omega_2^{\tilde\psi})$.
By Lemma \ref{inflation to a direct product}, this extension is
\[
0\ \to\ \dbZ/p\ \xrightarrow{\bar i\mapsto(\overline{pi},1)}\ (\dbZ/p^2)\times\tilde G'
\ \xrightarrow{(\tilde\psi\inv\circ\pr,\id)}\ \tilde G\times\tilde G'\ \to\ 1,
\]
where $\pr\colon\dbZ/p^2\to\dbZ/p$ is again the natural projection.
But the latter extension is equivalent to
\[
0\ \to\ \dbZ/p\ \xrightarrow{\bar i\mapsto(\overline{pi},0)}\ (\dbZ/p^2)\times(\dbZ/p)\
\xrightarrow{(\bar\psi\inv\circ\pr,(\bar\psi')\inv)}\ \bar G=\tilde G\times\tilde G'\ \to\ 1,
\]
which is $\omega_6^{(\bar\psi,\bar\psi')}$.

Now by Proposition \ref{extensions of cup}(f), $\bar\psi\cup\bar\psi'$ corresponds to  $\omega_4^{(\bar\psi,\bar\psi')}$.
It therefore follows from Lemma \ref{Baer example} that $\bar\psi\cup\bar\psi'+\beta_{\bar G}(\bar\psi)$
corresponds to $\omega_5^{(\bar\psi,\bar\psi')}$.
\end{proof}

\subsection*{C)\ Summary}
Putting together the results of the previous two subsections we obtain:

\begin{prop}
\label{summary1}
Let $\bar\alp\in\Omega(\bar G)$ have simple type and trivial kernel.
Then $\Lam_{\bar G}(\bar\alp)\in H^2(\bar G)$ corresponds to one of the following extensions:
\begin{enumerate}
\item[(i)]
when $p=2$: \quad $\omega_0$, $\omega_1^{\bar\psi}$,  $\omega_2^{\bar\psi}$, $\omega_3^{(\bar\psi,\bar\psi')}$;
\item[(ii)]
when $p\neq2$: \quad $\omega_0$, $\omega_1^{\bar\psi}$,  $\omega_2^{\bar\psi}$, $\omega_5^{(\bar\psi,\bar\psi')}$,
\end{enumerate}
where $\bar\psi,\bar\psi'$ are taken as above.
\end{prop}
\begin{proof}
When $p=2$ we have $\bar\alp=\bar\psi\tensor\bar\psi'$, with $\bar\psi,\bar\psi'\in H^1(\bar G)$
and $\Ker(\bar\psi)\cap\Ker(\bar\psi')=1$.
Furthermore, $\Lam_{\bar G}(\bar\alp)=\bar\psi\cup\bar\psi'$.
Now apply Proposition \ref{extensions of cup}.

When $p\neq2$ we write $\bar\alp=(\bar\psi\tensor\bar\psi',\bar\psi)$ where again
$\bar\psi,\bar\psi'\in H^1(\bar G)$ and $\Ker(\bar\psi)\cap\Ker(\bar\psi')=1$.
Here $\Lam_{\bar G}(\bar\alp)=\bar\psi\cup\bar\psi'+\beta_{\bar G}(\bar\psi)$.
If $\bar\psi=0$, then this corresponds to either $\omega_0$ or $\omega_1^{\bar\psi}$, by
Proposition \ref{extensions of cup}(a)(b).
If $\bar\psi\neq0$ and $\bar\psi,\bar\psi'$ are $\dbF_p$-linearly dependent,
then $\Lam_{\bar G}(\bar\alp)=\beta_{\bar G}(\bar\psi)$ corresponds to $\omega_2^{\bar\psi}$,
by Proposition \ref{extension of Bockstein}.
Finally, if $\bar\psi,\bar\psi'$ are $\dbF_p$-linearly independent,
then by Proposition \ref{extensions for linearly independent elements},
$\Lam_{\bar G}(\bar\alp)$ corresponds to $\omega_5^{(\bar\psi,\bar\psi')}$.
\end{proof}

One has the following converse result:

\begin{prop}
\label{summary2}
When $p=2$ let $i\in \{0,1,2,3\}$ and when $p\neq2$ let $i\in\{0,1,2,5\}$.
Let $(\dbZ/p)^s$ be the image of the epimorphism in $\omega_i$ (so $s=0,1,1,2,2$ for $i=0,1,2,3,5$, respectively).
Let $\theta\colon \bar G\xrightarrow{\sim}(\dbZ/p)^s$ be an isomorphism.
There exists $\bar\alp\in\Omega(\bar G)$ of simple type and with trivial kernel such that
$\Lam_{\bar G}(\bar\alp)\in H^2(\bar G)$ corresponds to $\omega_i^\theta$.
\end{prop}
\begin{proof}
We may assume that $\bar G=(\dbZ/p)^s$ and $\theta=\id$.
Let $\pr_j\colon(\dbZ/p)^2\to\dbZ/p$ be the projection on the $j$th coordinate, $j=1,2$.

When $p=2$ we take  $\bar\alp=\bar\psi\tensor\bar\psi'$, where
\[
(\bar\psi,\bar\psi')=(0,0),\  (\id_{\dbZ/2},0),\  (\id_{\dbZ/2},\id_{\dbZ/2}),\ (\pr_1,\pr_2),
\]
to obtain using Proposition \ref{extensions of cup}(a)(b)(c)(e)
$\omega_0$, $\omega_1$,  $\omega_2$, $\omega_3$, respectively.

When $p\neq2$ we take $\bar\alp=(\bar\psi\tensor\bar\psi',\bar\psi)$, where
$(\bar\psi,\bar\psi')=(0,0), \ (0,\id_{\dbZ/p})$,
to obtain using Proposition \ref{extensions of cup}(a)(b) the extensions $\omega_0$, $\omega_1$,  respectively.
Also, take $\bar\alp=(\bar\psi\tensor\bar\psi',\bar\psi)$, where
$\bar\psi=\bar\psi'=\id_{\dbZ/p}$ to obtain using Corollary \ref{extensions for linearly dependent elements}
the extension $\omega_2$.
Finally, $\bar\alp=(\bar\psi\tensor\bar\psi',\bar\psi)$, where $\bar\psi=\pr_1$, $\bar\psi'=\pr_2$,
gives using  Proposition \ref{extensions for linearly independent elements} the extension $\omega_5$.
\end{proof}

\section{Lifting of homomorphisms}

We now apply the computations of the previous section to solve some specific embedding
problems.

\begin{lem}
\label{embedding problem for Bockstein}
Let $G$ be a profinite group and $\psi\colon G\to\dbZ/p$ an epimorphism.
Then $\beta_G(\psi)=0$ if and only if $\psi$ factors via the natural map $\dbZ/p^2\to\dbZ/p$.
\end{lem}
\begin{proof}
Let $\bar G=G/\Ker(\psi)\isom\dbZ/p$ and let $\pi\colon G\to\bar G$ be the natural map.
There exists $\bar\psi\in H^1(\bar G)$ with $\psi=\bar\psi\circ\pi$ and $\inf_G(\bar\psi)=\psi$.
Then $\inf_G(\beta_{\bar G}(\bar\psi))=\beta_G(\psi)$.
By Proposition \ref{extension of Bockstein}, $\beta_{\bar G}(\bar\psi)$ corresponds to $\omega_2^{\bar\psi}$.
It follows from the last sentence of Proposition \ref{Hoechsmann} that $\beta_G(\psi)=0$ if and only if
the following embedding problem is solvable
\[
\xymatrix{
& G\ar[dl]_{\Phi}\ar[d]^{\psi}&\\
\dbZ/p^2\ar[r] & \dbZ/p\ar[r] &0.
}
\]
Note that if the homomorphism $\Phi$ exists, then it must be surjective.
\end{proof}

In the next proposition let $r,s$ be the generators of $M_{p^3}$ as in \S8.

\begin{prop}
\label{Zp2 or Mp3}
Let $p\neq2$ and let $G$ be a profinite group of Galois relation type.
Every epimorphism $\psi\colon G\to\dbZ/p$ factors through one of the epimorphisms:
\begin{enumerate}
\item[(i)]
the natural map $\dbZ/p^2\to\dbZ/p$;
\item[(ii)]
the map $\lam\tagg\colon M_{p^3}\to\dbZ/p$, defined by $r\mapsto\bar1$, $s\mapsto\bar0$.
\end{enumerate}
\end{prop}
\begin{proof}
If $\beta_G(\psi)=0$, then by Lemma \ref{embedding problem for Bockstein}, $\psi$ breaks via the map (i).

Next assume that $\beta_G(\psi)\neq0$.
Since $G$ has Galois relation type, there exists $\xi\in H^1(G)$ with $\psi\cup\xi+\beta_G(\psi)=0$.
In particular, $\psi\cup\xi\neq0$.
Since $p\neq2$ and the cup product is alternate, $\psi$ and $\xi$ are $\dbF_p$-linearly independent.
Now let $\bar G=G/(\Ker(\psi)\cap\Ker(\xi))\isom(\dbZ/p)^2$, and let $\pi\colon G\to\bar G$ be the canonical map.
Take $\bar\psi,\bar\xi\in H^1(\bar G)$ with $\psi=\bar\psi\circ\pi$, $\xi=\bar\xi\circ\pi$, $\inf_G(\bar\psi)=\psi$,
and $\inf_G(\bar\xi)=\xi$.
Then
\[
\textstyle{\inf_G}(\Lam_{\bar G}(\bar\psi\tensor\bar\xi,\bar\psi))=\Lam_G(\psi\tensor\xi,\psi)=0.
\]
By Proposition \ref{extensions for linearly independent elements},
$\Lam_{\bar G}(\bar\psi\tensor\bar\xi,\bar\psi)$ corresponds to $\omega_5^{(\bar\psi,\bar\xi)}$.
It follows again from Proposition \ref{Hoechsmann} that the embedding problem
\[
\xymatrix{
& G\ar[dl]_{\Phi}\ar[d]^{(\psi,\xi)}&\\
M_{p^3}\ar[r]^{\lam'} & (\dbZ/p)^2\ar[r] &0
}
\]
is solvable.
By Remark \ref{subgroups of Mp3}, no proper subgroup of $M_{p^3}$ is mapped surjectively by $\lam'$.
Therefore $\Phi$ is surjective.
As before, let $\pr_1\colon(\dbZ/p)^2\to\dbZ/p$ be the projection on the first coordinate.
We deduce that $\psi$ breaks via the epimorphism $\pr_1\circ\lam'$, which is just $\lam''$.
\end{proof}

Next let $r,s$ be the generators of $D_4$ as in \S8.
We write $G(p)$ for the maximal pro-$p$ quotient of the profinite group $G$.
One has the following analog of Proposition \ref{Zp2 or Mp3} for $p=2$.

\begin{prop}
\label{lifting in even case}
Let $p=2$ and let $G$ be a profinite group of Galois relation type and such that $G(2)\not\isom\dbZ/2$.
Every epimorphism $\psi\colon G\to\dbZ/2$ factors via one of the epimorphisms:
\begin{enumerate}
\item[(i)]
the natural map $\dbZ/4\to\dbZ/2$;
\item[(ii)]
the map $\rho'\colon D_4\to\dbZ/2$ , defined by $r\mapsto\bar1$, $s\mapsto\bar0$;
\item[(iii)]
the map $\rho\tagg\colon D_4\to\dbZ/2$, defined by $r\mapsto\bar0$, $s\mapsto\bar1$.
\end{enumerate}
\end{prop}
\begin{proof}
Let $\xi$ be as in Definition \ref{Galois relation type}(ii).
By the assumptions, $G(2)\neq1,\dbZ/2$.
Hence, if $G(2)$ is pro-cyclic, then $\psi$ factors via the map (i).
We may therefore assume that $G(2)$ is not pro-cyclic.

If $\beta_G(\psi)=0$, then by Lemma \ref{embedding problem for Bockstein}, $\psi$ factors
via the map (i).

Next we assume that $\psi,\xi$ are $\dbF_2$-linearly independent.
Let $\bar G=G/(\ker(\psi)\cap\Ker(\xi))$ and let $\pi\colon G\to \bar G$ be the natural map.
There exist $\bar\psi,\bar\xi\in H^1(\bar G)$ with $\psi=\bar\psi\circ\pi$, $\xi=\bar\xi\circ\pi$,
$\inf_G(\bar\psi)=\psi$ and $\inf_G(\bar\xi)=\xi$.
Note that $\Ker(\bar\psi)\cap\Ker(\bar\xi+\bar\psi)=\Ker(\bar\psi)\cap\ker(\bar\xi)=1$.
By Proposition \ref{extensions of cup}(e), $\bar\psi\cup(\bar\xi+\bar\psi)$ corresponds to the extension
$\omega_3^{(\bar\psi,\bar\xi+\bar\psi)}$.
By the choice of $\xi$ and Lemma \ref{formula for Bockstein when q is 2},
\[
\textstyle\inf_G(\bar\psi\cup(\bar\xi+\bar\psi))=\psi\cup\xi+\psi\cup\psi=\psi\cup\xi+\beta_G(\psi)=0.
\]
Proposition \ref{Hoechsmann} therefore implies that the embedding problem
\begin{equation}
\label{embed prob}
\xymatrix{
& G\ar@{->>}[ld]_{\Phi} \ar[d]^{(\psi,\xi+\psi)} &\\
 D_4 \ar[r]^{\rho} & (\dbZ/2)^2 \ar[r] & 0
}
\end{equation}
(with $\rho$ as in \S8) is solvable.
Since no proper subgroup of $D_4$ is mapped by $\rho$ surjectively onto
$(\dbZ/2)^2$ (Remark \ref{subgroups of Mp3}),  $\Phi$ is surjective.
We deduce that $\psi$ factors via the epimorphism $\pr_1\circ\rho$, which is just $\rho'$.

Finally assume that $\beta_G(\psi)\neq0$ and $\psi,\xi$ are $\dbF_2$-linearly dependent.
As $\beta_G(\psi)=\psi\cup\xi$, necessarily $\xi\neq0$.
But $\psi\neq0$, so $\psi=\xi$.

Now $G(2)$ is not pro-cyclic, so there exists $\psi'\in H^1(G)$ such that $\psi,\psi'$ are $\dbF_2$-linearly independent.
Then $\psi',\xi+\psi'$ are also $\dbF_2$-linearly independent.
By the argument above, the embedding problem (\ref{embed prob}), with $(\psi,\xi+\psi)$ replaced by $(\psi',\xi+\psi')$,
is solvable.
Composing with the map $\sig\colon(\dbZ/2)^2\to\dbZ/2$, $(\bar i,\bar j)\mapsto\overline{i+j}$, we
obtain that $\psi=\xi$ factors via $\sig\circ\rho$, which is just $\rho\tagg$.
\end{proof}

\section{The main results}

Let $G$ be a again a profinite group and $q=p$ a prime number.

\begin{thm}
\label{characterization of distinguished}
The following conditions on a normal open subgroup $N$ of $G$ are equivalent:
\begin{enumerate}
\item[(a)]
$N$ is distinguished;
\item[(b)]
\begin{enumerate}
\item[(i)]
When $p=2$, $G/N$ is isomorphic to one of the  groups
\[
1,\ \dbZ/2,\ (\dbZ/2)^2,\ \dbZ/4, \  D_4;
\]
\item[(ii)]
When $p\neq 2$, $G/N$ is isomorphic to one of the  groups
\[
1,\  \dbZ/p,\  (\dbZ/p)^2,\ \dbZ/p^2,\   M_{p^3}.
\]
\end{enumerate}
\end{enumerate}
\end{thm}
\begin{proof}
(a)$\Rightarrow$(b):\quad
Let $N$ be distinguished and $M$, $\bar\alp$, $\varphi$ data for $N$.
Thus $N=\Ker(\varphi)$.
Set $\bar G=G/M$ and consider a central extension
\[
\omega: \qquad 0\ \to\ \dbZ/p\ \to\ B\ \to\ \bar G\ \to\ 1
\]
corresponding to $\Lam_{\bar G}(\bar\alp)$.
By Proposition \ref{distinguished and embedding problems}, $G/N$ embeds in $B$.

If $p=2$, then by Proposition \ref{summary1}, $\omega$ is equivalent to an extension of one of the
forms $\omega_0$, $\omega_1^{\bar\psi}$,  $\omega_2^{\bar\psi}$, $\omega_3^{(\bar\psi,\bar\psi')}$.
Then $G/N$ embeds in one of the groups $\dbZ/2$, $(\dbZ/2)^2$, $\dbZ/4$, $D_4$, and is therefore as in (i).

If $p\neq2$, then by Proposition \ref{summary1}, $\omega$ is equivalent to an extension
of one of the forms $\omega_0$, $\omega_1^{\bar\psi}$,  $\omega_2^{\bar\psi}$, $\omega_5^{(\bar\psi,\bar\psi')}$.
Then $G/N$ embeds in one of the groups $\dbZ/p$, $(\dbZ/p)^2$, $\dbZ/p^2$, $M_{p^3}$, and is therefore as in (ii).

\medskip

(b)$\Rightarrow$(a):\quad
By Example \ref{example of distinguished}, $G$ itself is distinguished.
We may therefore assume that $G/N$ is nontrivial.
Hence it is isomorphic to the middle group $B$ of $\omega_i$ where $i\in \{0,1,2,3\}$,
if $p=2$, and $i\in\{0,1,2,5\}$, if $p\neq2$.
Therefore there is an open normal subgroup $M$ of $G$ such that $N\leq M$ and the following diagram commutes:
\[
\xymatrix{
\omega: &  0\ar[r] &\dbZ/p\ar[r]\ar@{=}[d] & G/N\ar[r]\ar[d]_{\wr} & G/M\ar[r]\ar[d]_{\wr}^{\theta} & 1\\
\omega_i: & 0\ar[r] & \dbZ/p \ar[r] & B\ar[r] & (\dbZ/p)^s \ar[r] & 0,
}
\]
where $\theta$ is an isomorphism.
Then $\omega,\omega_i^\theta$ are equivalent.
By Proposition \ref{summary2}, $\omega_i^\theta$ corresponds to $\Lam_{G/M}(\bar\alp)\in H^2(G/M)$
for some $\bar\alp\in \Omega(G/M)$ of simple type and with trivial kernel, and therefore so does $\omega$.
We conclude from Proposition \ref{distinguished and embedding problems} that $N$ is distinguished.
\end{proof}

We deduce the following stronger form of the Main Theorem:

\begin{cor}
\label{list p odd}
Suppose that $p\neq2$ and let $G$ be a profinite group of Galois relation type.
Then $G^{(3)}$ is the intersection of all normal open subgroups $N$ of $G$
with $G/N$ isomorphic to one of $1$, $\dbZ/p^2$, $M_{p^3}$.
\end{cor}
\begin{proof}
By Theorems \ref{formal main theorem} and \ref{characterization of distinguished},
$G^{(3)}$ is the intersection of all normal open subgroups $N$ of $G$ with $G/N$
 isomorphic to one of $1$, $\dbZ/p$,   $\dbZ/p^2$,  $M_{p^3}$.
By Proposition \ref{Zp2 or Mp3}, $\dbZ/p$ can be omitted from this list.
\end{proof}

For $p=2$  Theorem \ref{formal main theorem}, Theorem \ref{characterization of
distinguished} and Proposition \ref{lifting in even case} give:

\begin{cor}
\label{list p even}
Let $p=2$ and let $G$ be of Galois relation type.
Then $G^{(3)}$ is the intersection of all normal open subgroups $N$ of $G$
such that $G/N$ is isomorphic to one of the groups
$1$, $\dbZ/2$, $\dbZ/4$, $D_4$.
\end{cor}

By Proposition \ref{Galois nature}, this extends \cite{MinacSpira96}*{Cor.\ 2.18},
which proves it for $G=G_F$, $F$  a field.
Combined with Proposition \ref{lifting in even case} it gives:

\begin{cor}
\label{short list p even}
Let $p=2$ and let $G$ be a profinite group of Galois relation type such that $G(2)\not\isom\dbZ/2$.
Then $G^{(3)}$ is the intersection of all normal open subgroups $N$ of $G$
such that $G/N$ is isomorphic to one of the groups $1$, $\dbZ/4$,  $D_4$.
\end{cor}

\begin{rems}
\rm
(a)\quad
The converse of Corollary \ref{short list p even} also holds: if $G(2)\isom\dbZ/2$, then
$G^{(3)}$ is not an intersection as above.

\medskip

(b)\quad
Let $F$ be a field of characteristic $\neq2$ and let $G=G_F$.
Then $G(2)\isom\dbZ/2$ if and only if $F$ is a Euclidean field, i.e., the set $(F^\times)^2$
of all nonzero squares in $F$ is an ordering on $F$ (\cite{Becker74}, \cite{EfratBook}*{\S19.2}).
Therefore, by (a), the Euclidean fields are those fields for which the
group $\dbZ/2$ cannot be omitted from the list in Corollary \ref{list p even}.
\end{rems}

\section{The structure of $G/G^{(3)}$}
When $p=2$, and $G=G_F$ for a field $F$, the quotient $G/G^{(3)}$
is the \textbf{$W$-group} of $F$, studied in \cite{MinacSpira90}, \cite{MinacSpira96}, and  \cite{MaheMinacSmith04}.
It encodes much of the ``real" arithmetic structure of $F$.
We now give some restrictions on the group structure of $G/G^{(3)}$ also for $p$ odd.

\begin{prop}
\label{G/G3}
Let $G$ be a profinite group of Galois relation type with $G/G^{(3)}$ nonabelian.
\begin{enumerate}
\item[(a)]
If $p\neq2$, then $M_{p^3}$ is a quotient of $G/G^{(3)}$.
\item[(b)]
If $p=2$, then $D_4$ is a quotient of $G/G^{(3)}$.
\end{enumerate}
\end{prop}
\begin{proof}
By our assumption, $G^{(3)}$ cannot be an intersection of open normal subgroups $N$ of $G$ with $G/N$ abelian.
When $p\neq2$ (resp., $p=2$) Corollary \ref{list p odd}
(resp., Corollary \ref{list p even}) yields an open normal subgroup
$N$ of $G$ with $G/N\isom M_{p^3}$ (resp., $G/N\isom D_4$).
The natural epimorphism $h\colon G\to \bar G=G/N$ maps $G^{(3)}$ to $\bar G^{(3)}$, which is trivial by Remark \ref{comments on Mp3 D8}(b).
Hence $h$ induces an epimorphism $\bar h\colon G/G^{(3)}\to\bar G$.
\end{proof}

We recover the following known ``automatic realizations":

\begin{cor}
Suppose that $F$ is a field of characteristic $\neq p$ and containing a root of unity of order $p$.
\begin{enumerate}
\item[(a)]
$($\cite{Brattstrom89}$)$
If $p\neq2$ and $H_{p^3}$ is realizable as  a Galois group over $F$,
then $M_{p^3}$ is also realizable as a Galois group over $F$.
\item[(b)]
$($\cite{MinacSmith91}*{Prop.\ 2.1}$)$
If $p=2$ and $Q_8$ is realizable as a Galois group over $F$, then $D_4$ is also realizable over $F$.
\end{enumerate}
\end{cor}
\begin{proof}
When $p\neq2$ (resp., $p=2$), take $\bar G=H_{p^3}$ (resp., $\bar G=Q_8$).
Thus $\bar G$ is a quotient of $G=G_F$, and as $\bar G^{(3)}=1$ (by Remark \ref{comments on Mp3 D8}(b)), also of $G/G^{(3)}$.
Hence $G/G^{(3)}$ is nonabelian.
Now apply Proposition \ref{G/G3}.
\end{proof}

The next fact was earlier proved in \cite{BensonLemireMinacSwallow07}*{Th.\ A.3} when $G=G_F$ for a field $F$ containing a primitive $p$th root of unity.

\begin{prop}
\label{elements of order p}
Let $p\neq2$ and let $G$ be a profinite group of Galois relation type.
Every element of $G/G^{(3)}$ of order $p$ belongs to $G^{(2)}/G^{(3)}$.
\end{prop}
\begin{proof}
It suffices to show that the elements of order $p$ in $G/G^{(3)}$ are in the kernel of every epimorphism $\bar\psi\colon G/G^{(3)}\to \dbZ/p$.
Now $\bar\psi$ lifts to a unique epimorphism $\psi\colon G\to\dbZ/p$.
By Proposition  \ref{Zp2 or Mp3}, $\psi$ breaks via an epimorphism $\pi\colon\bar G\to\dbZ/p$,
where either $\bar G=\dbZ/p^2$ and $\pi$ is the natural projection, or $\bar G=M_{p^3}$ and $\pi=\lam\tagg$
(where $\lam\tagg$ maps the generators $r,s$ of $M_{p^3}$ to $\bar1,\bar0$, respectively).
In both cases, $\bar G^{(3)}=1$, by Remark \ref{comments on Mp3 D8}(b) again.
Therefore there is a commutative triangle
\[
\xymatrix{
&G/G^{(3)} \ar@{->>}[d]^{\bar\psi}\ar@{->>}[dl]\\
\bar G \ar@{->>}[r]^{\pi} & \dbZ/p.\\
}
\]
In both cases $\pi$ is trivial on elements of $\bar G$ of order $p$
(for $\bar G=M_{p^3}$ this follows from Remark \ref{comments on Mp3 D8}(c)).
The claim follows.
\end{proof}

\begin{rem}  \rm
Proposition \ref{elements of order p} is no longer true when $p=2$.
For instance, $G=\dbZ/2(\isom G_\dbR)$ has Galois relation type, yet $G/G^{(3)}=\dbZ/2$ and $G^{(2)}/G^{(3)}=1$.
More generally, take $G=G_F$ for a field $F$ of characteristic $\neq2$.
Then $G/G^{(3)}$ contains an involution which is not in $G^{(2)}/G^{(3)}$ if and only if $F$ is
formally real \cite{MinacSpira90}*{Th.\ 2.7}.
\end{rem}

\begin{exam}
\label{non W-groups}
\rm
Suppose that $p\neq2$ and that $G$ has Galois relation type.
By Proposition \ref{elements of order p}, $G/G^{(3)}$ cannot be isomorphic to $(\dbZ/p)^I$,
with $I\neq\emptyset$ to $H_{p^3}$, nor to $M_{p^3}$ (see Remark \ref{comments on Mp3 D8}(c)).
\end{exam}

\begin{rem}  \rm
By the celebrated Artin--Schreier theorem, an absolute Galois group of a field
is either $1$, $\dbZ/2$, or is infinite.
Our results provide a new cohomological proof of this fact in characteristic $0$, as follows.

Assume that $F$ is a field of characteristic $0$ with $G=G_F$ finite.
If $G\isom\dbZ/p$ with $p\neq2$, then $F$ contains a primitive $p$th root of unity.
By Proposition \ref{Galois nature}, $G$ has Galois relation type,
contrary to what we have seen in Example \ref{non W-groups}.
This shows that $G$ is a finite $2$-group.

Next suppose that $G$ contains an element of order $4$.
We may then assume that $G\isom\dbZ/4$.
Let $K$ be its unique subgroup of order $2$ and write $H^1(K)=\{0,\psi\}$.
The map $\res_K\colon H^1(G)\to H^1(K)$ is trivial.
Hence the Kummer element $\kappa_2(-1)\in H^1(K)$ (which comes from $H^1(G)$)
is zero.
By Proposition \ref{Bockstein and Kummer}, $\beta_K(\psi)=\psi\cup\kappa_2(-1)=0$.
On the other hand, there are no epimorphisms $K\to\dbZ/4$,
contrary to Lemma \ref{embedding problem for Bockstein}.

Hence, $G$ consists of involutions.
By Proposition \ref{lifting in even case}, $G\isom 1,\dbZ/2$.
\end{rem}

\section{Examples}
We first give examples showing that none of the groups listed in Corollaries \ref{list p even}
and \ref{list p odd} can be omitted from these lists.

\begin{exam} \rm
Taking $G=G_\dbC=1$ we see that the trivial group cannot be removed from the above lists.
\end{exam}

\begin{exam}  \rm
For $p=2$ and $G=G_\dbR=\dbZ/2$ one has $G^{(3)}=1$.
This shows that $\dbZ/2$ cannot be removed from the list in Corollary \ref{list p even}.
\end{exam}

\begin{exam}  \rm
Let $F$ be a finite field, and let $G=G_F=\hat\dbZ$.
Then $G/G^{(3)}\isom\dbZ/p^2$.
Therefore $\dbZ/p^2$ cannot be removed from the lists.
\end{exam}

\begin{exam}  \rm
Take $p=2$ and $F=\dbR((t))$.
Then $G=G_F=\langle\tau,\eps\ |\ \eps^2=(\tau\eps)^2=1\rangle$
\cite{EfratBook}*{\S22.1}.
There is an epimorphism $G\to D_4$, $\tau\mapsto r$, $\eps\mapsto s$ (with notation as in \S8).
Hence $G/G^{(3)}$ is non-abelian.
Now let $N_0$ be the intersection of all closed normal subgroups $N$ of $G$
such that $G/N$ is isomorphic to one of  $1$, $\dbZ/2$,and  $\dbZ/4$.
Then $G/N_0$ is abelian (in fact, it is isomorphic to $(\dbZ/2)^2$).
Consequently, $N_0\neq G^{(3)}$.
Therefore $D_4$ cannot be removed from the list in Corollary \ref{list p even}.
\end{exam}

\begin{exam}  \rm
Let $p\neq2$.
Dirichlet's theorem on primes in arithmetical progressions yields $n\geq0$ with
$l=p(pn+1)+1$ prime.
Let $\zeta_{p^2}$ be in the algebraic closure of $\dbF_l$.
Then $\dbF_l$ contains the $p$th roots of unity, but does not contain a $p^2$th root of unity $\zeta_{p^2}$.
Therefore the maximal pro-$p$ Galois group $G_{\dbF_l}(p)$ has a generator $\bar\sig$
such that $\bar\sig(\zeta_{p^2})=\zeta_{p^2}^{1+p}$.
Lift $\bar\sig$ to some $\sig\in G=G_{\dbQ_l}(p)$.
Also let $\tau$ be a generator of the inertia group of $G$.
Then $G$ is generated by $\tau$ and $\sig$, subject to the defining relation
$\sig\tau\sig\inv=\tau^{1+p}$  \cite{EfratBook}*{Example 22.1.6}.

Now let $N_0$ be the intersection of all closed normal subgroups $N$ of $G$
with $G/N$ isomorphic to $1$ or $\dbZ/p^2$.
Then $G/N_0$ is abelian.
Since there is an epimorphism $G\to M_{p^3}$, $\tau\mapsto r$, $\sig\mapsto s$ (notation as in \S8),  $G/G^{(3)}$ is non-abelian, so  $N_0\neq G^{(3)}$
(in fact, $G/N_0\isom(\dbZ/p^2)\times (\dbZ/p)$ while $G/G^{(3)}\isom (\dbZ/p^2)\rtimes(\dbZ/p^2)=
\langle\tilde\tau\rangle\rtimes\langle\tilde\sig\rangle$,
with action $\tilde\sig\tilde\tau\tilde\sig\inv=\tilde\tau^{1+p}$).
Thus $M_{p^3}$ cannot be removed from the list in Corollary \ref{list p odd}.
\end{exam}

Our final two examples show that in Corollaries \ref{list p even} and \ref{list p odd} one cannot
omit the assumption that $G$ has Galois relation type.

\begin{exam}  \rm
\label{false for Q8}
Let $p=2$ and let $G=Q_8$.
Then $G$ has no normal subgroups $N$ with $G/N\isom\dbZ/4$ or $G/N\isom D_4$, and
has three distinct normal subgroups $N$ with $G/N\isom \dbZ/2$, all containing the center $Z(G)$.
Thus the intersection of all normal subgroups $N$ of $G$ as in Corollary \ref{list p even}
is $Z(G)(\isom\dbZ/2)$.
On the other hand, $G^{(3)}=1$ (Remark \ref{comments on Mp3 D8}(b)).
\end{exam}

\begin{exam}  \rm
\label{false for Zp, Hp3}
Let $p\neq2$ and let $G=\dbZ/p$ or $G=H_{p^3}$.
Then $G$ has no quotients isomorphic to $\dbZ/p^2$ or $M_{p^3}$.
Thus the intersection in Corollary \ref{list p odd} is $G$ itself.
But by Remark \ref{comments on Mp3 D8}(b), $G^{(3)}=1$.
\end{exam}

In this respect, the Main Theorem is a genuine structural result about absolute Galois groups.

\begin{rem}
\label{cohomological computations}
\rm
In view of Corollaries \ref{list p even} and \ref{list p odd}, the previous two
examples show that $Q_8$ (when $p=2$) and $\dbZ/p$, $H_{p^3}$ (when $p\neq2$) do not have Galois relation type.
This can be seen directly as follows.

For $G=Q_8$ and $p=2$, one has a graded ring isomorphism
\[
H^*(Q_8)\isom\dbF_2[x,y,z]/(x^2+xy+y^2,x^2y+xy^2),
\]
where $x,y,z$ have degrees $1,1,4$, respectively (see \cite{Adem97}*{p.\ 811, Example},
\cite{AdemMilgram04}*{Ch.\ IV, Lemma 2.10}).
In this ring, no product of nonzero elements of degree $1$ vanishes, yet $x^2+xy+y^2=0$.
Hence condition (i) of Definition \ref{cup} is not satisfied for $G=Q_8$.

For $G=\dbZ/p$ and $p\neq2$ one has
$H^*(G)\isom\dbF_p[x,y]/(x^2)$,
where $x,y$ have degrees $1,2$, respectively, and (with the obvious abuse of notation)
$\beta_G(x)=y$ \cite{Evens91}*{\S3.2}.
Here $\cup\colon H^1(G)\times H^1(G)\to H^2(G)$ is the zero map, but $\beta_G(x)\neq0$.
Hence  (ii) of Definition \ref{cup} is not satisfied.

For $G=H_{p^3}$ and $p\neq2$, the structure of $H^*(G)$ is considerably more complicated,
and was computed by Leary \cite{Leary92}*{Th.\ 6 and Th.\ 7}.
Here as well, $\cup\colon H^1(G)\times H^1(G)\to H^2(G)$ is the zero map, but $\beta_G$
is nontrivial.
Therefore condition (ii) of Definition \ref{cup} is not satisfied.
\end{rem}

\end{document}